\documentclass[12pt]{amsart}
\usepackage{amssymb, amstext, amscd, amsmath}
\usepackage[all]{xy}
\makeatletter
\def\@cite#1#2{{\m@th\upshape\bfseries%
[{#1\if@tempswa{\m@th\upshape\mdseries, #2}\fi}]}}
\makeatother
\theoremstyle{plain}
\newtheorem{theorem}{Theorem}[section]
\newtheorem{corollary}[theorem]{Corollary}
\newtheorem{proposition}[theorem]{Proposition}
\newtheorem{lemma}[theorem]{Lemma}

\theoremstyle{definition}
\newtheorem{definition}[theorem]{Definition}
\newtheorem{example}[theorem]{Example}
\newtheorem{remark}[theorem]{Remark}
\theoremstyle{remark}
\newcommand{\bbC}{{\mathbb{C}}}
\newcommand{\bbD}{{\mathbb{D}}}
\newcommand{\bbI}{{\mathbb{I}}}

\newcommand{\bbN}{{\mathbb{N}}}

\newcommand{\bbT}{{\mathbb{T}}}
\newcommand{\bbZ}{{\mathbb{Z}}}

\newcommand{\A}{{\mathcal{A}}}
\newcommand{\B}{{\mathcal{B}}}
\newcommand{\C}{{\mathcal{C}}}
\newcommand{\D}{{\mathcal{D}}}

\newcommand{\G}{{\mathcal{G}}}
\renewcommand{\H}{{\mathcal{H}}}

\newcommand{\J}{{\mathcal{J}}}
\newcommand{\K}{{\mathcal{K}}}
\renewcommand{\L}{{\mathcal{L}}}

\newcommand{\N}{{\mathcal{N}}}
\renewcommand{\O}{{\mathcal{O}}}

\newcommand{\T}{{\mathcal{T}}}

\newcommand{\V}{{\mathcal{V}}}

\newcommand{\X}{{\mathcal{X}}}

  \newcommand{\cN}{{\mathcal{N}}}
	\newcommand{\cO}{{\mathcal{O}}}

  \newcommand{\cT}{{\mathcal{T}}}

\newcommand{\rC}{{\mathrm{C}}}

\renewcommand{\phi}{\varphi}
\newcommand{\upchi}{{\raise.35ex\hbox{\ensuremath{\chi}}}}

\newcommand{\fin}{\operatorname{fin}}

\newcommand{\Aut}{\operatorname{Aut}}

\newcommand{\id}{{\operatorname{id}}}

\newcommand{\Ind}{{\operatorname{Ind}}}

\newcommand{\spn}{\operatorname{span}}

\newcommand{\ca}{\mathrm{C}^*}
\newcommand{\cenv}{\mathrm{C}^*_{\text{env}}}
\newcommand{\cmax}{\mathrm{C}^*_{\text{max}}}

\newcommand{\sot}{\textsc{sot}}

\newenvironment{sbmatrix}{\left[\begin{smallmatrix}}
{\end{smallmatrix}\right]}

\newcommand{\sca}[1]{\left\langle#1\right\rangle}
\newcommand\cpr{\rtimes_{\alpha}^{r}\, {\mathcal{G}}}
\newcommand\cpf{\rtimes_{\alpha}\, {\mathcal{G}}}

\newcommand\cpd{ \hat{\rtimes}_{\alpha}\, {\mathcal{G}}}

\begin{document}
%%%%%%%%%%%%%%%%%%%%%%%%%%%%%%%%%%%%%%
\title{The non-selfadjoint approach to the Hao-Ng isomorphism}

\author[E.G. Katsoulis]{Elias~G.~Katsoulis}
\address {Department of Mathematics 
\\East Carolina University\\ Greenville, NC 27858\\USA}
\email{katsoulise@ecu.edu}

\author[C. Ramsey]{Christopher~Ramsey}
\address {Department of Mathematics and Statistics
\\MacEwan University \\ Edmonton, AB \\Canada}
\email{ramseyc5@macewan.ca}

%%%%%%%%%%%%%%%%
\begin{abstract}
In an earlier work, the authors proposed a non-self-\break adjoint approach to the Hao-Ng isomorphism problem for the full crossed product, depending on the validity of two conjectures stated in the broader context of crossed products for operator algebras. By work of Harris and Kim, we now know that these conjectures in the generality stated may not always be valid. In this paper we show that in the context of hyperrigid tensor algebras of $\ca$-correspondences, each one of these conjectures is equivalent to the Hao-Ng problem. This is accomplished by studying the representation theory of non-selfadjoint crossed products of C$^*$-correspondence dynamical systems; in particular we show that there is an appropriate dilation theory. A large class of tensor algebras of $\ca$-correspondences, including all regular ones, are shown to be hyperrigid. Using Hamana's injective envelope theory, we extend earlier results from the discrete group case to arbitrary locally compact groups; this includes a resolution of the Hao-Ng isomorphism for the reduced crossed product and all hyperrigid $\ca$-correspondences. A culmination of these results is the resolution of the Hao-Ng isomorphism problem for the full crossed product and all row-finite graph correspondences; this extends a recent result of Bedos, Kaliszewski, Quigg and Spielberg. 
\end{abstract}
%%%%%%%%%%%%%%%%


\maketitle

\section{Introduction}

Let $((X,\C), \G, \alpha)$ be a C$^*$-correspondence dynamical system where $\G$ is a locally compact group and $\alpha$ is a generalized gauge action. This action can be extended uniquely to the Cuntz-Pimsner algebra $\O_X$
The \textit{Hao-Ng isomorphism problem} asks whether 
\[
\O_X \rtimes_\alpha \G \simeq \O_{X\rtimes_\alpha \G}
\]
in the reduced or full crossed products. This problem is named after Hao and Ng who proved the validity of this formula when $\G$ is amenable \cite[Theorem 2.10]{HN}. However, this formula was first studied by Abadie in the context of Takai duality for equivalence bimodules. Indeed, in Abadie's proof for the Takai duality, the Hao-Ng isomorphism forms the crucial step of the proof and corresponds to the key isomorphism of \cite[Lemma 7.2]{Williams} in the classical case. In general, the Hao-Ng isomorphism has proved to be a significant stimulant to research as versions of it appear in many different contexts, e.g. in Schafhauser's work \cite{Sch} on AF-embedability, or in Deaconu's work \cite{Dea, DKQ} on group actions on graph $\ca$-algebras. In its full generality, the problem remains open and under investigation by several authors \cite{BKQR, KQR, KQR2, Kim, Mor}.

The authors initiated a study in \cite{KatRamMem, Kat, KatRamCP2} of non-selfadjoint crossed products of operator algebra dynamical systems $(\A, \G, \alpha)$ where $\alpha$ acts by completely isometric isomorphisms of $\A$. The main thrust of \cite[Chapter 7]{KatRamMem} and \cite{Kat} is that the Hao-Ng isomorphism problem can and should be thought of as a non-selfadjoint problem. For the reduced crossed product this kind of approach has been and continues to be quite successful. For instance, we now know that the Hao-Ng isomorphism for the reduced crossed product holds for all discrete groups \cite{Kat}, a fact that resolves an open problem from \cite{BKQR} (and more is accomplished in this paper). 

The Hao-Ng isomorphism for the \textit{full} crossed product seems to be a much harder problem. In \cite{KatRamMem} we envisioned the following line of attack. First one verifies 
\begin{equation}\label{env-iso}
\cenv(\A \rtimes_\alpha \G) \simeq \cenv(A) \rtimes_\alpha \G
\end{equation}
in the full crossed product case for an arbitrary non-selfadjoint dynamical system $(\A, \G, \alpha)$; this is Problem 1 in \cite{KatRamMem}. Subsequently, one solves Problem 2 in \cite{KatRamMem} by showing that all relative crossed products coincide. Assuming that both problems have been resolved in the positive, now one specializes on tensor algebra dynamical systems and obtains 
\begin{equation}\label{tensor-iso}
\T_X^+ \rtimes_\alpha \G \simeq \T^+_{X\rtimes_\alpha \G} 
\end{equation}
by invoking the solution of Problem 2 and the remarks following \cite[Theorem 7.13]{KatRamMem}. Recalling that $\cenv(\T_X^+) = \O_X$, one recovers now the Hao-Ng isomorphism by combining equations (\ref{env-iso}) and (\ref{tensor-iso}). Note that even though a positive answer for \textit{both} Problems 1 and 2 leads to a positive resolution for the Hao-Ng isomorshism, the exact relation of each one of these problems with the Hao-Ng isomorphism was never clarified in \cite{KatRamMem}.

The central result of this paper, Theorem~\ref{thm;summarize}, clarifies that relation and shows that the Hao-Ng problem actually leads to equivalent statements in non-selfadjoint operator algebra theory, whose validity or refutation will therefore resolve the isomorphism. Specifically, for a non-degenerate $\ca$-correspondence $X$ we show that validity of (\ref{env-iso}) for $\A=\T^+_X$ is equivalent to the validity of the Hao-Ng isomorphism $\O_X \rtimes_\alpha \G \simeq \O_{X\rtimes_\alpha \G}$. In addition, for a large class of $\ca$-correspondences, including all regular ones, we show that the validity of the Hao-Ng isomorphism $\O_X \rtimes_\alpha \G \simeq \O_{X\rtimes_\alpha \G}$ is equivalent to the fact that all relative crossed products for $(\T^+_X, \G, \alpha)$ coincide, where $\alpha$ is a generalized gauge action. (For a general $\ca$-correspondence $X$ this last statement is implied by any of the previous two.) 

Theorem~\ref{thm;summarize} relates to exciting new work by Harris and Kim \cite{HK}. Indeed these authors have answered both Problems 1 and 2 from  \cite[Chapter 7]{KatRamMem} by producing finite dimensional, hyperrigid dynamical systems $(\A, \G, \alpha)$ with distinct relative crossed products and failing (\ref{env-iso}). However the examples of Harris and Kim \cite{HK} do not concern tensor algebras of $\ca$-correspondences and so the Hao-Ng problem remains open. Theorem~\ref{thm;summarize} shows now that the resolution of the Hao-Ng problem will lead to or will follow from  the existence or the absence of Harris-Kim type examples but in the realm of tensor algebras. Needless to say that the quest for such examples, or the refutation of their existence, becomes now a project of high priority.

To test our new results, we study the Hao-Ng isomorphism for a class of $\ca$-correspondences that plays a central role in the theory: graph $\ca$-correspondences. In Theorem \ref{hao-ng} we show that the Hao-Ng isomorphism problem is true in the case of row-finite graph correspondences, thus showing that the crossed product of such a Cuntz-Krieger algebra is the Cuntz-Pimsner algebra of a crossed product (of a graph) correspondence. This is done by showing that in the case of a dynamical system $(\A, \G, \alpha)$ where $\A$ is the tensor algebra of \textit{any} graph, $\G$ any locally compact group and $\alpha$ a generalized gauge action, all relative crossed products coincide. Then, Theorem~\ref{thm;summarize} finishes the proof for row-finite graphs. Note that in the special case where $\G$ is discrete, Theorem \ref{hao-ng} has also been obtained independently by Bedos, Kaliszewski, Quigg and Spielberg using different methods~\cite[Corollary 6.8 and Remark 6.10]{BKQS}. It is worth mentioning here that Theorem \ref{hao-ng} is essentially obtained by dilating representations in a wholly constructive manner and should prove of much interest to those who study the representation theory of C$^*$-correspondences. At the moment, the lack of a constructive dilation proof of $\cenv(\T_X^+) = \O_X$ seems to be a barrier to establishing (\ref{tensor-iso}) for the full crossed product in general. 

On the way to proving the above theorems we obtain several results of independent interest. First we resolve Problem 3 from our monograph \cite[Chapter 8]{KatRamMem}. Specifically, we show if $(X, \C)$ is a non-degenerate $\ca$-correspondence and $\alpha: \G \rightarrow (X, \C)$ is the generalized gauge action of a locally compact group, then $\T_X^+ \cpf $ is necessarily the tensor algebra of some $\ca$-correspondence.

Another result of independent interest is Theorem~\ref{thm:hyperrigid}, which identifies a large class of hyperrigid $\ca$-correspondences, i.e., $\ca$-correspon-\break dences whose tensor algebras are hyperrigid. Indeed our central Theorem~\ref{thm;summarize} actually applies to all hyperrigid $\ca$-correspondences. To make that result usable, we show that any $\ca$-correspondence $(X, \C, \phi_X)$ with $\phi_X(\J_X)X=X$ is hyperrigid (here $\J_X$ denotes Katsura's ideal). This includes all previous known examples of hyperrigid $\ca$-correspondences and many more, e.g., all regular ones.

An interesting byproduct of our techniques on the full crossed product version of the Hao-Ng problem, is the resolution of the same problem for the reduced crossed product and all hyperrigid $\ca$-correspon-\break dences. In \cite{Kat} the first named author verified the Hao-Ng isomorphism for the reduced crossed product and all discrete groups. Because here we are addressing locally compact groups which may not be discrete, we have to use an approach different from that of \cite{Kat}. In particular, the algebra $\A$ does not embed in either $\A \cpf$ or $\A \cpr$ and so restricting a maximal map of the crossed product on the core algebra $\A$ (as we did in \cite{Kat}) is no longer an option. Instead we use Hamana's injective envelope theory, an approach towards the Hao-Ng isomorphism which is used in this paper for the first time. This approach was adopted after illuminating discussions with S. Echterhoff and we are grateful to him for that.

We denote by $\bbN$ the set of positive integers, while $\bbZ^+_0=\bbN \cup \{0\}$.  We denote by $\overline{\spn}\{\cdots\}$ the closure of the linear span of $\{\cdots\}$. An ideal of a $\ca$-algebra always means a closed two-sided ideal.

\section{Crossed products and C$^*$-covers}

Let $(\A, \G, \alpha)$ be an operator algebra dynamical system, meaning that $\A$ is an approximately unital operator algebra and $\G$ is a locally compact (Hausdorff) group acting continuously on $\A$ by completely isometric automorphisms, $\alpha: \G \rightarrow \Aut(\A)$. The aim of this section is to better understand the relationship of $\alpha$-admissible C$^*$-covers. Recall that a C$^*$-cover $(\C, \iota)$ of $\A$ is a C$^*$-algebra $\C$ and a complete isometry $\iota : \A \rightarrow \C$ such that $C^*(\iota(\A)) = \C$.

The two nicest C$^*$-covers of $\A$ are the ``biggest'' and the ``smallest'' covers $\cmax(\A)$ and  $\cenv(\A)$. These are defined by their universal properties. Namely, whenever $(\C, \iota)$ is a C$^*$-cover there are (unique) surjective $*$-homomorphisms $\varphi : \cmax(\A) \rightarrow \C$ and $\psi : \C \rightarrow \cenv(\A)$ such that $\varphi(a) = \iota(a)$ and $\psi(\iota(a)) = a$, for all $a\in \A$. 

From \cite{KatRamMem}, a C$^*$-cover $(\C,\iota)$ is called {\em $\alpha$-admissible} if there exists a group representation $\beta : \G \rightarrow \Aut(\C)$ acting on $\C$ by $*$-automorphisms such that 
\[
\beta_s(\iota(a)) = \iota(\alpha_s(a)), \quad \forall s\in \G, a\in \A.
\]
In \cite[Lemma 3.3]{KatRamMem} we established that both $\cenv(\A)$ and $\cmax(\A)$ are always $\alpha$-admissible.
However, in \cite{KatRamMem} we did not provide any examples of C$^*$-covers which fail to be $\alpha$-admissible. We thank David Sherman for bringing this to our attention and asking us whether such covers do exist.

\begin{proposition}
Not all C$^*$-covers are $\alpha$-admissible.
\end{proposition}
\begin{proof}

Let $\C = C(\bbT) \oplus M_2$ and $\iota : A(\bbD) \rightarrow \C$ be given by $z \mapsto z\oplus \begin{sbmatrix} 0&0\\ 1&0\end{sbmatrix}$. By von Neumann's inequality it is straightforward that $\iota$ is a complete isometry. Now 
\begin{align*}
\iota(z) - \iota(z^2)\iota(z)^* & = (z\oplus \begin{sbmatrix} 0&0\\ 1&0\end{sbmatrix}) - (z^2\oplus \begin{sbmatrix} 0&0\\ 0&0\end{sbmatrix})(\bar z \oplus \begin{sbmatrix} 0&1\\ 0&0\end{sbmatrix}) \\
& = 0 \oplus \begin{sbmatrix} 0&0\\ 1&0\end{sbmatrix}.
\end{align*}
Thus, $C^*(\iota(A(\bbD))) = \C$ and $(\C, \iota)$ is a C$^*$-cover of $A(\bbD)$.

Consider the M\"{o}bius transformation $\varphi(z) = \frac{z - \frac{1}{2}}{1-\frac{z}{2}}$ which gives $\varphi\in \Aut(\bbD)$. From this define the dynamical system $(A(\bbD), \alpha, \bbZ)$ where $\alpha_n(f) = f\circ\varphi^n$  which is the same as $z\mapsto \varphi^n(z)$. It is well known that composition with a M\"{o}bius map is a completely isometric automorphism of the disc algebra.

Suppose that there exists $\tilde\alpha : \bbZ \rightarrow \Aut(\C)$ such that $\tilde\alpha_n(i(f)) = i(\alpha_n(f)), \forall f\in A(\bbD)$.
Calculating
\begin{align*}
\varphi(z) = \frac{z - \frac{1}{2}}{1-\frac{z}{2}} & = \left(z- \frac{1}{2}\right)\left( 1 + \frac{z}{2} + \frac{z^2}{4} + \cdots \right) \\
& = -\frac{1}{2} + \frac{3}{4}z + \frac{3}{8}z^2+ \cdots
\end{align*}
we get that
\[
\iota(\varphi(z)) = \varphi(z) \oplus \begin{sbmatrix} -\frac{1}{2}&0\\ \frac{3}{4}&-\frac{1}{2}\end{sbmatrix}.
\]
Hence,
\begin{align*}
\tilde\alpha_1\left( 0 \oplus \begin{sbmatrix} 0&0\\ 1&0\end{sbmatrix} \right) & = \tilde\alpha_1(\iota(z) - \iota(z^2)\iota(z)^*)
\\ & = \iota(\alpha_1(z)) - \iota(\alpha_1(z^2))\iota(\alpha(z))^*
\\ & = \iota(\varphi(z)) - \iota(\varphi(z))^2 \iota(\varphi(z))^*
\\ & = \varphi(z) \oplus \begin{sbmatrix} -\frac{1}{2}&0\\ \frac{3}{4}&-\frac{1}{2}\end{sbmatrix} - \left(\varphi(z)^2 \oplus \begin{sbmatrix} \frac{1}{4}&0\\ -\frac{3}{4}& \frac{1}{4}\end{sbmatrix}\right)\left( \overline{\varphi(z)} \oplus \begin{sbmatrix} -\frac{1}{2}&\frac{3}{4}\\ 0&-\frac{1}{2}\end{sbmatrix}\right)
\\ & = 0 \oplus \begin{sbmatrix} -\frac{3}{8}&-\frac{3}{16}\\ \frac{3}{8}&\frac{3}{16}\end{sbmatrix}.
\end{align*}
But then
\[
\left\| 0 \oplus \begin{sbmatrix} -\frac{3}{8}&-\frac{3}{16}\\ \frac{3}{8}&\frac{3}{16}\end{sbmatrix} \right\|
= \sqrt{2} \left\| \begin{sbmatrix} \frac{3}{8}\\ \frac{3}{16}\end{sbmatrix} \right\| 
= \frac{3\sqrt{10}}{16} < 1 = \left\| 0 \oplus \begin{sbmatrix} 0&0\\ 1& 0\end{sbmatrix} \right\|,
\]
a contradiction as $*$-automorphisms are isometric. Therefore, no such $\tilde\alpha$ exists and $(\C, \iota)$ is a non $\alpha$-admissible C$^*$-cover of $(A(\bbD), \alpha, \bbZ)$.
\end{proof}

%\begin{theorem}
%Let $(\A, \G, \alpha)$ be a dynamical system. If $\A$ satisfies the semi-Dirichlet property inside the C$^*$-cover $(\C, j)$ then $\A \rtimes_{(\C,j),\alpha} \G$ is semi-Dirichlet.
%\end{theorem}
%\begin{proof}
%This follows from the exact same observation proving \cite[Corollary 5.9]{KatRamMem}.
%\end{proof}

If we do have an $\alpha$-admissible cover then we can abuse the notation and call the group representation $\alpha$ again because of the next result.

%%%%%%%
\begin{lemma}
Let $(\A, \G, \alpha)$ be an operator algebra dynamical system. If $(\C, \iota)$ is an $\alpha$-admissible C$^*$-cover then there is a unique group representation of $\G$ on $\C$ acting by $*$-automorphisms extending $\alpha$.
\end{lemma}
\begin{proof}
Let $\beta_1$ and $\beta_2$ be two such extensions. Then for any $s \in \G$ we have $\beta_{1,s}\circ\beta_{2, s}^{-1} = \id$ on $\A$ and we just need to prove this is the identity map.
To this end assume that $\beta : \G \rightarrow \Aut(\C)$ extends the identity map on $\A$. That is, $\beta_s|_\A = \id_\A$, for all $s\in\G$.

By the universal property of $\cmax(\A)$ there is a unique surjective $*$-homomorphism $\varphi : \cmax(\A) \rightarrow \C$ such that $\varphi(a) = \iota(a)$, for all $a\in\A$. Notice that $\beta_s\circ\varphi(a) = \beta_s\circ\iota(a) = \iota(a)$, for all $a\in\A$. Thus, $\beta_s\circ\varphi = \varphi$ by uniqueness which implies that $\beta_s= \id$ on $\C$, for any $s \in \G$.
\end{proof}

Now we turn to crossed products of non-selfadjoint operator algebras. 

\begin{definition}[\cite{KatRamMem}]
Let $(\A, \G, \alpha)$ be an operator algebra dynamical system and let $(\C,\iota)$ be an $\alpha$-admissible C$^*$-cover. The {\em relative reduced and full crossed products} are denoted by $\A \rtimes^r_{(\C,\iota),\alpha} \G$ and $\A \rtimes_{(\C,\iota),\alpha} \G$ and are defined to be the closure of $C_c(\G, \A)$ in $\C \rtimes^r_\alpha \G$ and $\C \rtimes_\alpha \G$, respectively.
\end{definition}

All relative reduced crossed products are in fact completely isometrically isomorphic \cite[Theorem 3.12]{KatRamMem} and so we define the {\em reduced crossed product}, denoted $\A \rtimes^r_\alpha \G$, to be this unique object. Lastly, we define the {\em full crossed product} to be 
\[
\A \rtimes_\alpha \G := \A \rtimes_{\cmax(\A), \alpha} \G.
\]
In fact, this is the universal algebra for all covariant representations of $(\A, \G, \alpha)$ \cite[Proposition 3.7]{KatRamMem}. Finally, it should be noted that, as in the selfadjoint case, if $\G$ is amenable then the full and reduced crossed products coincide \cite[Theorem 3.14]{KatRamMem}.

Now we are able to state and prove the main theorem of this section. 

%%%%%%%%%
\begin{theorem}\label{thm::minmax}
Let $(\A, \G, \alpha)$ be an operator algebra dynamical system. Then for every $\alpha$-admissible C$^*$-cover $(\C, \iota)$ there are surjective completely contractive homomorphisms 
\[
\A \rtimes_\alpha \G \xrightarrow{q_{max}} \A \rtimes_{(\C,\iota), \alpha} \G \xrightarrow{q_{min}} \A \rtimes_{\cenv(\A), \alpha} \G
\]
such that they are just the identity on $C_c(\G, \A)$.
\end{theorem}

\begin{proof}
Label the unique extensions of $\alpha$ to the C$^*$-covers $\cenv(\A), (\C, \iota)$ and $\cmax(\A)$ respectively $\alpha_{\rm env}, \alpha_{\C}$ and $\alpha_{\max}$.
By the universal properties there exists surjective $*$-homomorphisms 
\[
\varphi_{\rm env} : \C \rightarrow \cenv(\A)\ \ \textrm{and} \ \ \varphi_{\max} : \cmax(\A) \rightarrow \C
\]
such that $\varphi_{\rm env}(\iota(a)) = a$ and $\varphi_{\max}(a) = \iota(a)$, for all $a\in \A$.

By uniqueness of the quotient maps and of the extensions we have the following commutative diagram:
\[
\begin{CD}
\cmax(\A) @>\varphi_{\max}>> \C @>\varphi_{\rm env}>> \cenv(\A) \\
@V\alpha_{\max}VV @V\alpha_{\C}VV @V\alpha_{\rm env}VV \\
\cmax(\A) @>\varphi_{\max}>> \C @>\varphi_{\rm env}>> \cenv(\A)
\end{CD}
\]
Thus, $\ker \varphi_{\max}$ is an $\alpha_{\max}$-invariant ideal and $\ker \varphi_{\rm env}$ is an $\alpha_{\C}$-invariant ideal. By \cite[Proposition 3.19]{Williams}, full C$^*$-crossed products preserve exact sequences by $\alpha$-invariant ideals. Hence, we have the following surjective $*$-homomorphisms
\[
\cmax(\A) \rtimes_{\alpha_{\max}} \G \ \xrightarrow{\varphi_{\max}  \rtimes \id} \ \C \rtimes_{\alpha_{\C}} \G \ \xrightarrow{\varphi_{\rm env}\rtimes \id} \ \cenv(\A) \rtimes_{\alpha_{\rm env}} \G.
\]
So 
\[
q_{\max} = \varphi_{\max} \rtimes \id|_{\A \rtimes_\alpha \G} \  \textrm{and} \ q_{\min} = \varphi_{\rm env} \rtimes \id|_{\A \rtimes_{(\C, \iota), \alpha} \G}
\]
are completely contractive homomorphisms which amount to the identity on $C_c(\G, \A)$.
\end{proof}

The benefit of this theorem, as will be used later, is that one needs only to show that the map $q_{\min}\circ q_{\max}$ is a completely isometric isomorphism to establish that all relative crossed products are the same.

%%%%%%%%%%%%%%%%%%%%%%%%%
\section{Hyperrigidity and the Hao-Ng isomorphism}

A not necessarily unital operator algebra $\A$ is said to be \textit{hyperrigid} if given any (non-degenerate) $*$-homomorphism 
\[
\tau \colon \cenv(\A) \longrightarrow B(\H)
\]
then $\tau$ is the only completely positive, completely contractive extension of the restricted map  $\tau_{\mid \A}$.  By adding an injective direct summand if necessary, it is easy to see that in order to verify hyperrigidity, one needs to consider only injective  $*$-representations $\tau$ but this need not concern us here. The term hyperrigid was coined by Arveson in \cite{Arv} but the concept had been floating around in various forms before this, e.g. \cite{Duncan}. 

Our definition is slightly weaker than that of Duncan's \cite[Section 4]{Duncan} as Duncan requests that $\tau$ be the only completely contractive extension of the restricted map, i.e., no requirement of positivity in the non-unital case. In any case \cite[Proposition 4]{Duncan} shows that the graph algebra of any row-finite graph is hyperrigid. Actually we are about to provide a much stronger result but first we need to remind the reader the definition and some of the basic notation regarding $\ca$-correspondences.

A $\ca$- correspondence $(X,\C,\phi_X)$, or just $(X,\C)$, consists of a $\ca$-algebra $\C$, a Hilbert $\C$-module $(X, \sca{\phantom{,},\phantom{,}})$ and a
(non-degenerate) $*$-homomorphism $\phi_X\colon \C \rightarrow \L(X)$ into the C$^*$-algebra of adjointable operators on $X$. Equivalently, a (represented) $\ca$-correspondence $(X,\C)$ consists of a $\ca$-algebra $\C \subseteq B(\K)$, $\K$ a Hilbert space, and a norm-closed $\C$-bimodule $X \subseteq B(\K)$ satisfying $X^*X \subseteq \C$ (this  allows us to define the inner product  $\sca{\phantom{,},\phantom{,}}$) and $\overline{\spn}\{ \C X\} =X$ (this is the non-degeneracy of the left action of $\C$). The equivalence of the two definitions follows from the fact that an abstract $\ca$-correspondence embeds in the Toeplitz $\ca$-algebra $\T_X$ that we will define below and therefore can be represented on a Hilbert space. 

A representation $(\rho,t)$ of a $\ca$-correspondence into
$B(\mathcal H)$, is a pair consisting of  a non-degenerate $*$-homomorphism $\rho\colon \C \rightarrow B(\H)$ and a linear map $t\colon X \rightarrow B(\H)$, such that
 $$\rho(c)t(x)=t(\phi_X(c)(x)),$$
for all $c\in \C$ and $x\in X$. 
It is called an isometric (Toeplitz) representation when 
 $$t(x)^*t(x')=\rho(\sca{x,x'}),$$
for all $c \in \C$ and $x,x'\in X$. 

By the relations above, the $\ca$-algebra generated by an isometric representation $(\rho,t)$ equals the closed linear span of $$t(x_1)\cdots  t(x_n)t(y_1)^*\cdots t(y_m)^*, \quad x_i,y_j\in X.$$
For any
isometric representation $(\rho,t)$ there exists a $*$-homomorphism
$\psi_t:\K(X)\rightarrow B$, such that $\psi_t(\theta_{x,y})=
t(x)t(y)^*$, where $\K(X)$ is the subalgebra of $\L(X)$ of so-called compact operators generated by $\theta_{x,y}(z) = x \langle y,z\rangle$. (See \cite[Chapter 3]{Kat1} for more details on this topic.)

There exists a universal Toeplitz representation, denoted as $(\rho_{\infty} , t_{\infty})$, so that any other representation of $(X,\C)$ is equivalent to a direct sum of sub-representations of $(\rho_{\infty} , t_{\infty})$. The Cuntz-Pimsner-Toeplitz $\ca$-algebra $\T_X$ is defined as the $\ca$-algebra generated by the image of $(\rho_{\infty} , t_{\infty})$. 

The \emph{tensor algebra} $\T_{X}^+$ of a $\ca$-correspondence \cite{MS}
$(X,\C)$ is the norm-closed subalgebra of $\T_X$ generated by
all elements of the form $\rho_{\infty}(c), t_{\infty}(x)$, $c \in \C$, $x \in X$. The tensor algebra $\T_X^+$ contains a faithful copy of the $\ca$-correspondence $(X, \C)$. Thus $X$ inherits an operator space from $\T_X^+$; we can now say that a representation $(\rho, t)$ of $(X, \C)$ is completely contractive whenever $t$ is a completely contractive map with respect to that operator space structure. 
 
Consider the ideal $$\J_X\equiv  \phi_X^{-1}(\K(X))\cap \ker\phi_X^{\perp}.$$ (which we will call Katsura's ideal.) An isometric representation $(\rho, t)$ of $(X, \C,\phi_X)$ is said to be \textit{covariant (Cuntz-Pimsner)} if and only if $\psi_t ( \phi_X(c)) = \rho (c)$, for all $c \in \J_X$. The universal $\ca$-algebra for all isometric covariant representations of $(X, \C)$ is the Cuntz-Pimsner algebra $\O_X$. The algebra $\O_X$ contains (a faithful copy of) $\C$ and (a unitarily equivalent) copy of $X$.

The first author and Kribs \cite[Lemma 3.5]{KatsoulisKribsJFA} have shown that the non-selfadjoint algebra of $\O_X$ generated by these copies of $\C$ and $X$ is completely isometrically isomorphic to $\T_X^+$. Furthermore, $\cenv(\T_X^+)\simeq \O_X$. See \cite{KatsoulisKribsJFA, MS} for more details.

Now to the hyperrigidity of tensor algebras.

\begin{theorem}\label{thm:hyperrigid}
Let $(X, \C)$ be a $\ca$-correspondence. If $\phi_X(\J_X)$ acts non-degenerately on $X$, then $(X, \C)$ is a hyperrigid $\ca$-correspondence, i.e., $\T^+_X$ is a hyperrigid operator algebra.
\end{theorem}

\begin{proof}
Let $\tau\colon \O_X \longrightarrow B(\H)$ be a $*$-homomorphism and let $\tau'\colon \O_X  \longrightarrow B(\H)$ be a completely contractive and completely positive map that agrees with $\tau$ on $\T^+_{X}$. We are to prove that $\tau'$ is multiplicative and so it agrees with $\tau$. Since $\tau'$ is a completely contractive and completely positive map, we can use multiplicative domain arguments \cite[Proposition 1.5.7]{BRO}.

Let $(\rho, t)$ be the universal Cuntz-Pimsner representation of $(X, \C)$. Since $\rho(\C) \subseteq \T^+_X$ is a $\ca$-algebra, the multiplicative domain of $\tau'$ contains $\rho(\C)$. We claim that it also contains $t(X)$.

Indeed, for any $x \in X$ we have
\begin{equation} \label{one side}
\begin{aligned}
\tau'(t(x))^*\tau'(t(x))&=\tau(t(x))^*\tau(t(x))=\tau(t(x)^*t(x))\\
&=\tau(\rho(\sca{x, x}))=\tau'(\rho(\sca{x, x}))\\
&=\tau'(t(x)^*t(x)),
\end{aligned}
\end{equation}
where the equation on the second line holds because $\rho(\C) \subseteq \T_X^+$ and the two maps agree there.

Let $a \in \J_X$ and $x \in X$. Since $\phi_X(a) \in \K(X)$, we have $z_{m, k}, w_{m, k} \in X$, $m, k \in \bbN$, so that 
\begin{equation} \label{eq;fun0}
\phi_X(a) = \lim_{m \rightarrow \infty} \sum_{k} \, \theta_{z_{m, k}, w_{m,k}}
\end{equation}
is a limit of finite rank operators in $\K(X)$. Let $X_0\subseteq X$ be the $\C$-submodule generated by $x$ and all $z_{m, k}, w_{m, k} \in X$, $m, k \in \bbN$. Since $X_0$ is countably generated, Kasparov's Stabilization Theorem implies the existence of $\{x_n \}_{n=1}^{\infty}$ in $X_0$ so that $\|\sum_{n=1}^{\l} \theta_{x_n, x_n}\| \leq 1 $, for all $l \in \bbN  $, and
\begin{equation*}
\sum_{n=1}^{\infty} \theta_{x_n, x_n}(\xi) = \xi, \mbox{ for all } \xi \in X_0.
\end{equation*}
From this, a standard approximation argument involving (\ref{eq;fun0}) shows that 
\begin{equation} \label {eq;fun1}
 \sum_{n=1}^{\infty} \theta_{x_n, x_n} \phi_X(a) = \sum_{n=1}^{\infty} \phi_X(a) \theta_{x_n, x_n} =\phi_X(a),
\end{equation} 
with the convergence in the norm topology\footnote{This is exactly the same argument one uses on non-separable Hilbert space to write any compact operator as a (perhaps infinite) sum of rank-one operators.}. Then 
\begin{align*}
\phi_X(aa^*) &=\lim_k  \phi_X(a) \big(\sum_{n=1}^{k} \theta_{x_n, x_n}\big)\phi_X(a)^*\\
                    &=  \sum_{n=1}^{\infty} \theta_{\phi_X(a)x_n, \phi_X(a)x_n}.
\end{align*}
By the Schwarz inequality 
\begin{equation} \label{eq Dun1}
\tau'\big(t(\phi_X(a)x_n)t(\phi_X(a)x_n)^*\big) \geq \tau'(t(\phi_X(a)x_n))\tau'(t(\phi_X(a)x_n))^*,
\end{equation} for all $n \in \bbN$, and so 
\begin{align*}
\tau'(\rho(aa^*))&=\tau'(\psi_t(\phi_X(aa^*)) \\
&=\sum_{n=1}^{\infty}\tau'\big(t(\phi_X(a)x_n)t(\phi_X(a)x_n)^*\big) \\
&\geq \sum_{n=1}^{\infty}\tau'(t(\phi_X(a)x_n))\tau'(t(\phi_X(a)x_n))^*\\
&=\sum_{n=1}^{\infty}\tau(t(\phi_X(a)x_n))\tau(t(\phi_X(a)x_n))^*\\
&=\tau(\psi_t(\phi_X(aa^*))=\tau(\rho(aa^*)) = \tau'(\rho(aa^*)).
\end{align*}
Hence (\ref{eq Dun1}) is actually an equality. Combining this with (\ref{one side}), we conclude that $t(\phi_X(a)x_n)$ belongs to the multiplicative domain of $\tau'$, for all $a \in \J_X$ and $n \in \bbN$. Since $\rho(\C)$ is also contained in the multiplicative domain of $\tau'$, we have that 
\[
t(\phi_X(a)x)= \sum_{n=1}^{\infty} t(\phi_X(a)x_n)\rho(\sca{x_n,x})
\]
belongs to the multiplicative domain of $\tau'$, for all $a \in \J_X$ and $x \in X$. Since $\phi_X(\J_X)$ acts non-degenerately on $X$, the multiplicative domain of $\tau'$ contains $t(X)$, as desired. This completes the proof.\end{proof}

%\begin{remark} 
%In earlier versions of this paper circulated at the Arxiv, Theorem~\ref{thm:hyperrigid} was given a proof under the additionl assumption that $X$ is countably generated as a right $\C$-module. The current proof is a minor modification of the earlier proofs.
%\end{remark}

Recall that a $\ca$-correspondence $(X, \C)$ is said to be \textit{regular} iff $\C$ acts faithfully on $X$ by compact operators, i.e., $\J_X= \C$. The following is immediate.

\begin{corollary} \label{regular hyper}
A regular $\ca$-correspondence is necessarily hyperrigid.
\end{corollary}

We are about to see that the assumption of injectivity cannot be removed from the Corollary above.   But first we need criterion for the failure of hyperrigidity.

\begin{proposition} \label{prop;nonhyper}
Let $(X, \C)$ be a $\ca$-correspondence with $\J_X=\{0 \}$. Then $(X, \C)$ fails to be hyperrigid.
\end{proposition}

\begin{proof}
Let $(\pi, t)$ be any isometric representation of $(X, \C)$ on a Hilbert space $\H$. If $V_1, V_2$ are the unilateral and bilateral (forward) shift respectively, 
then the associations
\begin{equation} 
\begin{aligned}
\C \ni a &\longrightarrow a \otimes I, \\
X \ni x &\longrightarrow x \otimes V_i, \, i=1,2,
\end{aligned}
\end{equation}
determine isometric representations of $X$, which are neccesarilly Cuntz-Pimsner covariant, since $\J_X=\{0 \}$. Therefore they promote to representations $\phi_1$ and $\phi_2$ of $\O_X\simeq\cenv(\T_X^+)$ on $\H\otimes \ell^2(\bbZ^+_0)$ and $\H\otimes \ell^2(\bbZ)$ respectively. Now notice that when $\phi_2$ is being compressed on $\H\otimes \ell^2(\bbZ^+_0)$, it produces a completely positive contractive map $\tilde{\phi_2}\neq \phi_1$, which however agrees with $\phi_1$ on $\T_X^+$. Hence $(X, \C)$ is not hyperrigid.
\end{proof}

Recall that if $\alpha$ is an endomorphism of a $\ca$-algebra $\A$, then the semicrossed product  $\A\rtimes_{\alpha} \bbZ^+_0$ (also denoted as $\A\rtimes_{\alpha} \bbZ^+$ in the literature) is simply the tensor algebra of the $\ca$-correspondence $\A_{\alpha}$, where the left action on $\A$ is coming from $\alpha$. In the case where both $\A$ and $\alpha$ are unital and $\alpha$ is injective, such algebras are always hyperrigid. This has already been noted in the literature, eg. \cite{Kak}, but it is also an immediate consequence of Corollary~\ref{regular hyper}. It is worth noting that the requirement of $\alpha$ being injective cannot be dropped from neither Corollary \ref{regular hyper} nor the discussion above.

\begin{example} Let $\X$ be a compact Hausdorff space which is not a singleton and consider some $x \in X$ which is not an isolated point. Let $\phi: \X \rightarrow \X$ with $\phi(y) = x$, for all $y \in \X$. Then the semicrossed product $C(\X)\rtimes_{\phi}\bbZ^+_0$ is not hyperrigid.

Indeed, in that case, the kernel of the right action equals $C_0(\X\backslash \{x\})$. Hence Katsura's ideal is trivial and Proposition~\ref{prop;nonhyper} applies.
\end{example}

Finally recall that the C*-envelope of a non-unital operator algebra can be computed from the C*-envelope of its unitization. More precisely, as the pair $(\cenv(\A), \iota)$ where $\cenv(\A)$ is the C*-subalgebra generated by $\iota(\A)$ inside the C*-envelope $(\cenv(\A^1), \iota)$ of the (unique) unitization $\A^1$ of $\A$. By the proof of \cite[Proposition 4.3.5]{BlLM} this C*-envelope of an operator algebra $\A$ has the desired universal property, that for any C*-cover $(\iota', \B')$ of $\A$, there exists a (necessarily unique and surjective) $*$-homomorphism $\pi:\B' \to \cenv(\A)$, such that $\pi \circ \iota' = \iota$.

We start with an elementary result regarding crossed products.

\begin{lemma} \label{silly}
Let $(\C, \G, \alpha)$ be a $\ca$-dynamical system and let $\D\subseteq \C$ be the $\ca$-subalgebra of $\C$ generated by some selfadjoint approximate unit for $\C$. Then
\begin{itemize}
\item[(i)] $\C C_c(\G, \D)$ is dense in $\C \cpf$.
\item[(ii)] If $\pi : \C \rtimes_\alpha \G \rightarrow B(\H)$ is a non-degenerate representation, then its restriction on $ C_c(\G, \D)$ is also non-degenerate.
\end{itemize} 
\end{lemma}

\begin{proof}
Let $\{ e_i\}_{i \in \bbI}$ be the selfadjoint approximate unit generating $\D$. Then any elementary tensor $h \otimes c \in C_c(\G, \D)$, where $(h \otimes c )(s) = h(s)c$, $h \in \rC_c(\G)$, $c \in \C$, can be written as 
\[
h \otimes c = \lim_{i \in \bbI} c \left(h \otimes e_i \right)\in \overline{\C C_c(\G, \D)}.
\]
This implies (i). For (ii) notice that by taking adjoints in (i), $ C_c(\G, \D)\C$ is also dense in $\C \cpf$. Hence 
\[
\pi\left( C_c(\G, \D) \right)\H = \pi\left( C_c(\G, \D) \right)\pi(\C)\H=\pi\left( C_c(\G, \D)\C \right)\H 
\]
which is dense in $\pi(\C)\H= \H$ and the conclusion follows.
\end{proof}

Our next result has been established for all discrete groups in \cite{Kat}. Here we extend it to arbitrary locally compact groups provided that the pertinent algebras are hyperrigid. One of the key ingredients of the proof is the use of injectivity for operator spaces. We briefly review the key definitions and the results used in the proof. We follow \cite{Paulsen} in our presentation; most of the material first appeared in \cite{Hamana}.

An operator space $I$ is said to be injective provided that for any pair of operator spaces $E \subseteq F$ and completely contractive map $\phi : E \rightarrow I$, there exists a completely contractive map $\psi : F \rightarrow I$ that extends $\phi$. 

Given an operator space $F$, we say that $(E, \kappa)$ is an injective envelope of $F$ provided that
\begin{itemize}
\item[(i)] $E$ is injective,
\item[(ii)] $\kappa : F \rightarrow E$ is a complete isometry,
\item[(iii)] if $E_1$ is injective with $\kappa (F)\subseteq E_1 \subseteq E$, then $E_1=E$.
\end{itemize}

Hamana essentially showed that every operator space $F\subseteq B(\H)$ admits an injective envelope $ (E, \kappa)$, with $E \subseteq B(\H)$ and $\kappa $ being the inclusion map \cite[Theorem 15.4]{Paulsen}.  The proof of \cite[Theorem 15.4]{Paulsen} shows that $E$ materializes as the range of a completely contractive idempotent $\phi: B(\H) \rightarrow B(\H)$. If $B(\H) \ni I \in F$ then the completely contractive idempotent $\phi$ is unital and therefore completely positive. Hence the range of $\phi$, i.e., $E$, is an operator system and the Choi-Effros Theorem \cite[Theorem 15.2]{Paulsen} applies: setting $a\circ b =\phi(ab)$ defines a multiplication on $E = \phi(B(\H))$, and $E$ equipped with this multiplication and its usual $*$-operation becomes a $\ca$-algebra. If on top of beiing unital, $F$ happens to be an operator algebra as well, then the $\ca$-subalgebra of $(E, \circ)$ generated by $F$, gives the $\ca$-envelope of $F$ \cite[Theorem 15.16]{Paulsen}.

\begin{theorem} \label{hyperenv}
Let $\A$ be a hyperrigid operator algebra which possesses a contractive approximate unit $\{e_i\}_{i \in \bbI}$ consisting of selfadjoint operators. Let $\alpha \colon \G \rightarrow \Aut \A$ be a continuous action of a locally compact group. Then 
\begin{equation} \label{redequiv}
\cenv\left( \A \cpr \right)\simeq \cenv(\A) \cpr 
\end{equation}
and
 \begin{equation} \label{fullequiv}
 \cenv\big(\A\rtimes_{\cenv(\A), \alpha} \G \big) \simeq \cenv(\A) \cpf
 \end{equation}
 via canonical embeddings.
\end{theorem} 
\begin{proof}

Let $\rho \colon \cenv(\A) \rightarrow B(\H)$ be a faithful (non-degenerate) representation and let 
\begin{align*}
\tilde{\rho}  &\colon \cenv(\A) \longrightarrow B\big(\H \otimes L^2(\G)\big)\\
u &\colon\G \longrightarrow B\big(\H \otimes L^2(\G)\big)
\end{align*}
so that $\tilde{\rho} \rtimes u$ (which we will denote as $\pi$) is the regular representation induced by $\rho$. (See \cite[Section 2.2]{Williams} for notation and additional information.) Since $\rho$ is non-degenerate,  \cite[Lemma 2.17]{Williams} implies that the induced representation $\pi= \tilde{\rho} \rtimes u$ is also non-degenerate. 

Let 
\[
\phi \colon B\big(\H \otimes L^2(\G)\big) \longrightarrow B\big(\H \otimes L^2(\G)\big)
\]
be a completely contractive idempotent map whose range is the injective envelope of $\pi\left(\A \cpr\right)^1$.  Let $\D$ be the closed (selfadjoint) subalgebra of $\A$ generated by $\{e_i\}_{i \in \bbI}$. Then $C_c(\G, \D)$ is a selfadjoint subalgebra of $\A\cpr$ and so \cite[1.3.12]{BlLM} implies that 
\begin{equation} \label{doubleuse}
\phi\left(S\pi(f)\right) = \phi(S)\pi(f), 
\end{equation}
for any $S \in B\big(\H \otimes L^2(\G)\big)$ and $f  \in C_c(\G, \D)$. In particular 
\[
\phi\big (\tilde{\rho}(a)\pi(f)\big) = \phi(\tilde{\rho}(a))\pi(f), 
\]
for all $a \in \A$, $f \in C_c(\G, \D)$. On the other hand, 
\[
\tilde{\rho}(a)\pi(f) = \pi( af)\in \pi(\A \cpr)
\]
and so 
$\phi (\tilde{\rho}(a)\pi(f) ) = \tilde{\rho}(a)\pi(f)$, $a \in \A$, $f \in C_c(\G, \D)$. Hence 
\[
\big( \phi(\tilde{\rho}(a)) - \tilde{\rho}(a)\big) \pi(f) = 0, \mbox{ for all }f \in C_c(\G, \D). 
\]
By Lemma \ref{silly}(ii), $\pi \big( C_c(\G, \D)\big)$ acts non-degenerately on $\H \otimes L^2(\G)$ and so
\begin{equation} \label{identitymap}
\phi(\tilde{\rho}(a)) = \tilde{\rho}(a), \mbox{ for all } a \in \A.
\end{equation}
Hence the mapping $\phi$ is a completely positive and completely contractive extension of the identity map on $\tilde{\rho}(\A)$. However $\tilde{\rho}(\A)$ is hyperrigid and according to the discussion in the beginning of the section, the identity map on $\tilde{\rho}(\A)$ is the only such completely contractive and completely positive extension to $\cenv(\tilde{\rho}(\A)) = \tilde{\rho}\big(\cenv(\A)\big)$. Therefore 
\begin{equation*} 
\phi(\tilde{\rho}(c)) = \tilde{\rho}(c), \mbox{ for all }c \in \cenv(\A). 
\end{equation*}
Appealing again to (\ref{doubleuse}), with $S = \tilde{\rho}(c)$, we obtain
\[
\phi(\pi(cf))= \phi\big (\tilde{\rho}(c)\pi(f)\big) = \phi(\tilde{\rho}(c))\pi(f)= \tilde{\rho}(c)\pi(f) = \pi(cf), 
\]
for all $c \in \cenv(\A)$ and $f \in C_c(\G, \D)$. By Lemma \ref{silly}(i) we have 
\[
\phi(S) = S, \mbox{ for all } S \in \pi\big(\cenv(\A) \cpr\big).
\]
 But this implies that the Choi-Effros multiplication on  $$\pi\big(\cenv(\A) \cpr\big)^1 \subseteq \phi\big(B\big(\H \otimes L^2(\G)\big)\big)$$ is actually the original one coming from $B\big(\H \otimes L^2(\G)\big)$ and so the $\ca$-algebra generated by $$\pi (\A \cpr)^1 \subseteq \phi\big(B\big(\H \otimes L^2(\G)\big)\big)$$ equals $\pi \big(\cenv(\A) \cpr\big)^1$. Hence 
 \[
 \cenv\left(\pi (\A  \cpr )^1\right) = \pi \big(\cenv(\A) \cpr\big)^1. 
 \]
Furthermore, the $\ca$-algebra generated by $\pi(\A \cpr) \subseteq \pi \big(\cenv(\A) \cpr\big)^1$ equals $\pi \big(\cenv(\A) \cpr\big)$. This establishes (\ref{redequiv}).

In order to prove (\ref{fullequiv}), let this time $\pi := \tilde{\rho}\rtimes u$, where $(\tilde{\rho},  u)$ is the universal covariant representation of $(\cenv(\A), \G, \alpha)$. With this $\pi$, a verbatim repetition of the proof of (\ref{redequiv}) establishes (\ref{fullequiv}).
 \end{proof}
 
 Now we turn to crossed product correspondences. Let $\G$ be a locally compact group acting on a non-degenerate $\ca$-correspondence $(X,\C)$ by a generalized gauge action $\alpha : \G \rightarrow \Aut(\T_X)$, i.e., $\alpha_s(X) = X$ and $\alpha_s(\C) = \C$, for all $g \in \G$. The reduced crossed product correspondence $(X\cpr, \C\cpr)$ is the completion of $C_c(\G, X)$ and $C_c(\G, \C)$ in $\T_X^+ \cpr$, which can be thought of as living in $\T^+_X \cpr$ but equivalently can be considered as living in $\O_X \cpr$. The left and right module actions are given by multiplication and $\langle S,T\rangle = S^*T$, for $S,T\in C_c(\G, X)$.
 
 In a similar manner, one defines the full crossed product correspondence $(X\cpf, \C \cpf)$ by completing the spaces in $\T_X \cpf$. This was shown to be unitarily equivalent to the abstract characterization of the full crossed product correspondence in \cite[Remark 7.8]{KatRamMem}. Lastly, we recall the definition of the crossed product correspondence $(X \cpd, \C \cpd)$ which is the completion of the spaces in $\O_X \cpf$. In general, it is unknown whether these two correspondences are unitarily equivalent or not. 

Our next result has been established in \cite{KatRamMem} in the case where $\G$ is discrete. In \cite{KatRamMem} it was also noted that the proof carries over to the general locally compact case. We are about to explain how this is done. In the proof we will use the language of product systems, which we now discuss briefly. 

Let $G$ be a countable group with unit $e \in G$ and let $P \subseteq G$ be a positive cone. A product system $X = \{ X_p\}_{p \in P}$ over $(G, P)$ consists of a $\ca$-algebra $X_e \subseteq B(\K)$, $\K$ a Hilbert space, and a family of (represented) $X_e$-correspondences $X_p \subseteq B(\K)$, $p \in P\backslash\{e\}$, satisfying the semigroup rule $\overline{\spn}\{X_pX_q\}=X_{pq}$, for all $p,q \in P$. For instance, if $(X, \C)$ is a (represented) $\ca$-correspondence, then by taking $G=\bbZ$, $P = \bbZ^+_0$, $X_0=\C$ and $X_n =\overline{\spn}\{X^n\}$, $n =1, 2, \dots$, we obtain a product system $\{X_n\}_{n=0}^{\infty}$ over $(\bbZ, \bbZ^+_0)$. As with $\ca$-correspondences, one can define product systems abstractly but we will not do that here. See \cite{F} for more details.

If $X = \{ X_p\}_{p \in P}$ is a poduct system over $(G, P)$, then an isometric (Toeplitz) representation $\psi=\{\psi_p\}_{p \in P}$ of $X$ on a Hilbert space $\H$ consists of a $*$-representation $\psi_e: \X_e\rightarrow B(\H)$ and isometric $\ca$-correspondence representations $\psi_p :X_p \rightarrow B(\H)$, $p \in P$, so that 
$\psi_{pq}(x_p x_q)= \psi_p(x_p)\psi_q(x_q)$ for all $x_p\in X_p$, $x_q \in X_q$, $p,q\in P$. If $(\rho, t)$ is a representation of a $\ca$-correspondence $(X, \C)$ on $\H$, then by setting $\psi_0=\rho$ and 
\[
\psi_n: X_n \longrightarrow B(\H); x_1 x_2\dots x_n\longmapsto t(x_1)t(x_2)\dots t(x_n)
\]
we obtain a representation of the product system $\{X_n\}_{n=0}^{\infty}$ discussed in the previous paragraph. These representations are exactly the compactly aligned, Nica covariant representations of $\{X_n\}_{n=0}^{\infty}$ in the language of \cite{F} and so the Toeplitz $\ca$-algebra of $\{X_n\}_{n=0}^{\infty}$ generated by the theory of \cite{F} coincides with our $\T_X$. 

The use of the product system language in the proof  of the next result allows us to import an important result from the theory of product systems, Fowler's Theorem~\cite[Theorem 7.2]{F}. This result has no analogue within the theory of $\ca$-correspondences as it allows us to check whether a representation of the Toeplitz algebra of a $\ca$-correspondence is faithful without using gauge actions. 

\begin{theorem} \label{HNtensor2}
Let $\G$ be a locally compact group acting by a generalized gauge action $\alpha$ on a non-degenerate $\ca$-correspondence $(X, \C)$. Then
\[
\T^+_{X} \cpr \simeq \T^+_{X\cpr}.
\]
Therefore, $$\cenv\big (\T^+_{X} \cpr \big) \simeq  \O_{X\cpr}.$$
\end{theorem}

\begin{proof} Let $\rho \colon \cT_X \rightarrow B(\H)$ be some faithful $*$-representation and let $V \in B(\ell^2(\bbZ^+_0))$ be the forward shift.  As we did earlier, set $X_n :=X^n$ for $n\geq 1$, $X_0:= \C$ and for the rest of the proof let $X$ denote the product system $\{X_n\}_{n =0}^{\infty}$. Then the map
\[
X_n\ni x \longmapsto \rho(x) \otimes V^n \in B(\H \otimes l^2(\bbZ^+_0)), \quad n\in \bbZ^+_0,
\]
is a representation of $X$ that satisfies the requirements of Fowler's Theorem~\cite[Theorem 7.2]{F}. Therefore it establishes a faithful representation $\pi : \cT_X \rightarrow B(\H \otimes \ell^2(\bbZ^+_0))$.

Since $X\cpr \subseteq \cT_X \cpr$, we may consider the regular representation $\Ind_{\pi}$, when restricted on $X\cpr$, as a representation of the product system $X\cpr$, which we denote as $\psi$. We furthermore write $\psi_n := \psi\mid_{X_n\cpr}$, $n\in \bbZ^+_0 $.

We claim that $\psi$ satisfies the requirements of Fowler's Theorem~\cite[Theorem 7.2]{F}. Indeed let $Q_n\in l^2(\bbZ^+_0)$ be the projection on the one dimensional subspace corresponding to the characteristic function of $n \in \bbZ^+_0$ and let $\hat{Q}_n \equiv (I\otimes Q_{n})\otimes I \in B\big( \H \otimes l^2(\bbZ^+_0)\otimes L^2(\G)\big)$ be the (constant) $B(\H \otimes l^2(\bbZ^+_0))$-valued function that assigns the value $I\otimes Q_{n}$ to any $s \in \G$. Then given any $f \in C_c(\G, X_n)$, $n \geq1$, we have 
\begin{align*}
\big(\hat{Q}_0 \Ind_{\pi} f \big)h(t) &= \hat{Q}_0 \int_{\G} \pi\big(\alpha_{t}^{-1}(f(s))\big)h(s^{-1}t)ds\\
&=(I\otimes Q_{0})\otimes I \int_{\G} \big(\rho(\alpha_{t}^{-1}(f(s))) \otimes V^n\big)h(s^{-1}t)ds\\
&=\int_{\G} \big(\rho(\alpha_{t}^{-1}(f(s))) \otimes Q_0V^n\big)h(s^{-1}t)ds = 0
\end{align*}
Therefore if $P_n^{\psi}$ denotes the range space of $\psi(X_n\cpr)$, $ n \in \bbZ^+_0$, then the product $\prod_{n\in F\backslash\{0\}} (I - P_n^{\psi})$, which in our case collapses to a single factor of the form $I -P_n^{\psi}$, always dominates $\hat{Q}_0$ and so 
\begin{equation} \label{normeq}
\big\|\psi_0(f) \prod_{n\in F\backslash\{0\}} (I- P_n^{\psi})\big\| \geq\big\| \psi_0(f) \hat{Q}_0\big\| =\big\|\Ind_{\pi}(f)\hat{Q}_0 \big\|, 
\end{equation}
for any $f \in \C_c(\G, \C)$. However, each $I\otimes Q_n$ reduces $\pi\mid_{\C}$ and therefore $$\pi \mid_{\C}\simeq \big(\oplus_{n\in \bbZ^+_0}(I \otimes Q_n) \pi \big) \mid_{\C} \simeq \big(\oplus (I \otimes Q_0) \pi \big)\mid_{\C}, $$ i.e., the restriction of $\pi$ on $\C$ is unitarily equivalent to a direct sum of countably many copies of $(I \otimes Q_0) \pi $ restricted on $\C$. From this we obtain,
\[
\psi_0 = {\Ind_{\pi}}\mid _{\C \cpr} \ \ \simeq \ \oplus {\Ind_{ Q_0\pi}}\mid _{\C \cpr} \ \ \simeq \ \oplus {\hat{Q}_0\Ind_{\pi }}\mid _{\C \cpr}. 
\]
Combining the above with (\ref{normeq}) we now obtain 
\[
\big\|\psi_0(f) \prod_{n\in F\backslash\{0\}} (I - P_n^{\psi})\big\| =\big\|\Ind_{\pi}(f)\hat{Q}_0\big\| = \|\psi_0(f) \|
\]
for any $f \in \C_c(\G, \C)$, which establishes the claim.

Since the claim is valid, Fowler's Theorem shows now that the induced representation $\psi_{*}$ is a faithful representation of $\T_{X\cpr}$. Note now that $\psi_{*}(\T^+_{X\cpr}) \simeq \T^+_{X\cpr}$ is equal to the closed linear span of
\[
\psi_{*}(X\cpr)=\bigcup_{n \in \bbZ^+_0} \psi_n(X_n\cpr)=\bigcup_{n \in \bbZ^+_0}\overline{\Ind_{\pi}C_c(\G, X_n)}.
\]
However, $\T^+_X\cpr$ is also isomorphic to the closed linear span of  $$\bigcup_{n \in \bbZ^+_0}\overline{\Ind_{\pi}C_c(\G, X_n)}.$$ Hence, $\cT^{+}_X \cpr \simeq \cT^+_{X\cpr}$.

Finally $$\cenv\big (\cT^+_X \cpr \big) \simeq  \cenv(\cT^+_{X\cpr}) \simeq \cO_{X\cpr}.$$ 
with the last identification following from \cite[Theorem 3.7]{KatsoulisKribsJFA}.
\end{proof}

\begin{remark}
The use of the language of product systems in the previous proof has an additional benefit. By switching from $(\bbZ, \bbZ^+_0)$ to an arbitrary totally ordered group $(G, P)$, the same exact proof as above establishes the more general result $\cN\T^+_{X} \cpr \simeq \cN\T^+_{X\cpr}$, where $\N\T^+_X$ denotes the Nica tensor algebra of $X$. Actually the proof works for any quasi-lattice ordered group  $(G,P)$ provided that certain issues involving compact alignment and Nica covariance are worked out first. We are recording this fact here for future reference. 
\end{remark}

In \cite{Kat} the Hao-Ng isomorphism problem was resolved for the reduced crossed product and all discrete groups. In the next result we address the case of an arbitrary locally compact group and we resolve the Hao-Ng problem for the reduced crossed product provided that the $\ca$-correspondence is hyperrigid. Note that our result subsumes an earlier result \cite[Proposition 5.5]{Kim} which was posted on the arXiv but has not appeared in print.

\begin{theorem} \label{HaoNgreduced}
Let $\G$ be a locally compact group acting by a generalized gauge action $\alpha$ on a non-degenerate hyperrigid $\ca$-correspondence $(X, \C)$, e.g. $\phi_X(\J_X)X = X$. Then
\[
\O_X \cpr \simeq\O_{X\cpr}.
\]
  \end{theorem}
  
  \begin{proof} 
  By Theorem~\ref{HNtensor2} we have $\cenv\big (\T^+_{X} \cpr \big) \simeq  \O_{X\cpr}$. On the other hand, Theorem~\ref{hyperenv} implies that $\cenv\big (\T^+_{X} \cpr \big) \simeq \cenv (\T^+_{X}) \cpr  $. Hence $\O_X \cpr \simeq \cenv (\T^+_{X}) \cpr  \simeq \O_{X\cpr}$. 
  %Finally \cite[Theorem 3.7]{KatsoulisKribsJFA} shows that $\cenv (\T^+_{X}) \simeq \O_X$ and we are done.
  \end{proof}

It is important to us that an analogous results holds for the full crossed product.

 \begin{theorem} \label{HNtensor1}
Let $\G$ be a locally compact group acting by a generalized gauge action $\alpha$ on a non-degenerate hyperrigid $\ca$-correspondence $(X, \C)$. Then
\[
\O_{X } \rtimes_{\alpha} \G  \simeq  \O_{X\cpd}.
\]
 \end{theorem}

\begin{proof}
In \cite[Theorem 7.13]{KatRamMem} we proved that 
 \[
\T^+_{X} \rtimes_{\O_X ,\alpha} \G \simeq \T^+_{X\cpd}\quad \mbox{  and  } \quad \cenv\big (\T^+_{X } \rtimes_{\O_X ,\alpha} \G \big) \simeq  \O_{X\cpd}.
 \]
 Now (\ref{fullequiv}) in Theorem~\ref{hyperenv} and the above imply that 
 \[
  \O_{X\cpd}\simeq  \cenv\big (\T^+_{X}  \rtimes_{\cenv(\T^+_x) ,\alpha} \G  \big) \simeq \cenv (\T^+_{X}) \cpf \simeq \O_{X } \rtimes_{\alpha} \G
 \]
and the conclusion follows.
\end{proof}

At this point one might think that the above theorem is the final word regarding the Hao-Ng isomorphism for hyperrigid correspondences. As it turns out, this couldn't be further from the truth. It is indeed the case that we have expressed the crossed product $\O_{X } \rtimes_{\alpha} \G $ as the Cuntz-Pimsner algebra of a $\ca$-correspondence, namely $X\cpd$, but this is not the $\ca$-correspondence that the authors of \cite{BKQR} ask for. Is this a big deal? Most definitely yes, and we devote the next section explaining the reasons why.

%%%%%%%%%%%%

\section{Isometric coextensions}

The goal of this section is to answer Problem 3 in Chapter 8 of \cite{KatRamMem}: Is $\T_X^+ \rtimes_\alpha \G$ the tensor algebra of some $\ca$-correspondence? The (affirmative) answer is one of the key ingredients in the proof of Theorem~\ref{thm;summarize}, one of the central results of the paper.

\begin{definition}
Let $(X, \C)$ be a $\ca$-correspondence and $\G$ a locally compact group. A \textit{generalized gauge action} $\alpha: \G \rightarrow \Aut((X,\C))$ is a map from $\G$ into the completely isometric module automorphisms. In particular, for each $s\in \G$, $\alpha_s$ is an isometric automorphism of $X$ and a $*$-automorphism of $\C$ such that 
\[
\alpha_s(\xi c) = \alpha_s(\xi)\alpha_s(c), \ \ \alpha_s(\varphi_X(c)\xi) = \varphi_X(\alpha_s(c))\alpha_s(\xi)
\]
\[
\textrm{and} \ \ \alpha_s(\langle \xi, \eta\rangle) = \langle \alpha_s(\xi), \alpha_s(\eta)\rangle
\]
for all $\xi, \eta\in X$ and $c\in \C$.
\end{definition}

Earlier we said that $\alpha : \G \rightarrow \Aut \T_X$ forms a \textit{generalized gauge action of $\T_X$} if $\alpha_s(X) = X$ and $\alpha_s(\C) = \C$, for all $g \in \G$. The following result says that the two definitions are equivalent and it was observed in \cite[pg 5760]{Kat}. (See also \cite[Lemma 2.6]{HN} for the analogous result with $\O_X$).

%%%%%%
\begin{proposition}\label{prop::ggauge}
Let $(X,\C)$ be a non-degenerate $\ca$-correspondence and $\G$ a locally compact group.
If $\alpha : \G \rightarrow \Aut(\T_X)$ is a generalized gauge action of $\T_X$ then it restricts to a generalized gauge action of $(X,\C)$. Conversely, a generalized gauge action $\alpha$ of $(X,\C)$ extends uniquely to a generalized gauge action of $\T_X$.
\end{proposition}

The fundamental object of study in this section is the $\ca$-correspon-\break dence dynamical system $((X,\C), \G,\alpha)$ which is given by a non-degenerate $\ca$-correspondence $(X,\C)$ and a generalized gauge action $\alpha$ of the locally compact group $\G$ acting on $(X,\C)$.

\begin{definition}
A representation of the $\ca$-correspondence dynamical system $((X,\C), \G, \alpha)$ is a quadruple $(\rho, t, u, \H)$ consisting of a completely contractive representation $(\rho, t, \H)$ of $(X,\C)$ and a strongly continuous unitary representation $u: \G \rightarrow U(\H)$ satisfying the covariance relations
\[
u(s)t(\xi) = t(\alpha_s(\xi))u(s) \ \ \textrm{and} \ \ u(s)\rho(c) = \rho(\alpha_s(c))u(s)
\]
for all $s\in \G, \xi\in X$ and $c\in \C$. Moreover, $(\rho,t, u,\H)$ is said to be \textit{isometric} if $(\rho,t, \H)$ is an isometric (Toeplitz) representation of $(X,\C)$.
\end{definition}

The following theorem is an extension of \cite[Theorem 2.12]{MS}.

%%%%%
\begin{theorem}\label{thm::isometricreps}
Let $((X,\C),\G,\alpha)$ be a $\ca$-correspondence dynamical system.
The isometric representations $(\rho, t, u,\H)$ of $((X,\C),\G,\alpha)$ are in bijective correspondence with the isometric representations $(\pi, u, \H)$ of $(\T_X, \G, \alpha)$. Specifically, they are related by $\pi = \rho\rtimes t$.
\end{theorem}

\begin{proof}
If $(\pi, u, \H)$ is an isometric  representation of $(\T_X,\G, \alpha)$ then \cite[Theorem 2.12]{MS} proves that there exists an isometric  representation $(\rho,t,\H)$ of $(X,\C)$ such that $\rho\rtimes t = \pi$. Proposition \ref{prop::ggauge} gives that $\alpha$ is a generalized gauge action of $\G$ on $(X,\C)$ and so the covariance relations between $\rho, t$ and $u$ are automatic. Hence, $(\rho, t, u, \H)$ is an isometric  representation of $((X,\C),\G, \alpha)$.

Conversely, suppose $(\rho, t, u,\H)$ is an isometric  representation of\break $((X,\C),\G,\alpha)$. Again by \cite[Theorem 2.12]{MS} this gives that $(\rho \rtimes t, \H)$ is an isometric  representation of $\T_X$ and by Proposition \ref{prop::ggauge} $\alpha$ extends uniquely to a generalized gauge action of $\G$ on $\T_X$. Because $u(s)(\rho\rtimes t(\cdot)) u(s)^*$ and $(\rho\circ\alpha_s)\rtimes (t\circ\alpha_s) = (\rho\rtimes t)\circ\alpha_s$ agree on $X$ and $\C$ then the uniqueness of \cite[Theorem 2.12]{MS} gives that
\[
u(s)\rho\rtimes t(a) = \rho\rtimes t(\alpha_s(a))u(s)
\]
for all $s\in \G$ and $a\in \T_X$. Therefore, $(\rho\rtimes t, u, \H)$ is an isometric  representation of $(\T_X,\G,\alpha)$.
\end{proof}

If $(\rho, t, u, \H)$ and $(\rho_1, t_1, u_1, \H_1)$ are completely contractive  representations of $((X,\C), \G, \alpha)$ we say that the latter is a {\em dilation} of the former if $\H \subseteq \H_1$ and
\begin{enumerate}
\item $\H$ reduces $\pi_1$ and $u_1$ with $\pi_1(c)|_\H = \pi(c), c\in \C$ and $u_1(g)|_\H = u(g), g\in \G$, and
\item $\H$ is a semi-invariant subspace for $t_1$ and $P_\H t_1(\xi)|_\H = t(\xi), \xi \in X$.
\end{enumerate}

We call such a dilation an {\em extension} if $\H$ is an invariant subspace for $t_1$ and a {\em coextension} if $\H$ is a coinvariant subspace for $t_1$.

The dilation $(\rho_1, t_1, u_1, \H_1)$ is called {\em minimal} when $\H_1$ is the smallest reducing subspace for $t_1$ containing $\H$.

Now we need a lemma relating to the step-by-step dilation techniques of Muhly and Solel of \cite[Section 3]{MS}. Their proof is modelled after Popescu's step-by-step dilation technique \cite{Pop0} using the Scha\"{e}ffer matrix construction \cite{Schae}. The following is also a slight simplification of the original proof in \cite{MS}.

To this end suppose $(\rho, t, u, \H)$ is a completely contractive  representation of the $\ca$-correspondence dynamical system $((X,\C),\G,\alpha)$. The ultimate goal is to prove that every such representation dilates to an isometric  representation.

As in \cite{MS}, define the Hilbert space $\H^X = \overline{X \otimes_\rho \H}^{\langle\cdot,\cdot\rangle}$ where 
\[
\xi c \otimes h = \xi \otimes \rho(c)h \ \ \textrm{and} \ \ 
\langle \xi \otimes h, \eta \otimes k\rangle = \langle h, \rho(\langle \xi,\eta\rangle)k\rangle
\]
for $\xi,\eta\in X$, $c\in\C$ and $h,k\in \H$. As well, define $\sigma^X : X \rightarrow B(\H, \H^X)$ by $\sigma^X(\xi)h = \xi\otimes h$ and $\tilde t: \H^X \rightarrow \H$ by $\tilde t(\xi\otimes h) = t(\xi)h$. 
From here one defines the one step dilation to $\H_1 = \H \oplus \H^X$ given by
\[
t_1(\xi) = \left[\begin{array}{cc} t(\xi) & 0 \\ (I - \tilde t^* \tilde t)^{1/2} \sigma^X(\xi) & 0 \end{array}\right]
\]
and 
\[
\rho_1(c) = \left[\begin{array}{cc} \rho(c) & 0 \\ 0 & \tilde\rho(c) \end{array}\right]
\]
where $\tilde\rho : \C \rightarrow B(\H^X)$ is given by $\tilde\rho(c)(\xi\otimes h) = \varphi_X(c)\xi \otimes h$.

%%%%%
\begin{lemma}\label{lemma::onestep}
Consider
 \[
u_1(s) = \left[\begin{array}{cc} u(s) & 0 \\ 0 & \tilde u(s)\end{array}\right]
\]
where $\tilde u : G \rightarrow B(\H^X)$ is given by $\tilde u(s)(\xi\otimes h) = \alpha_s(\xi)\otimes u(s)h$, which is well-defined.
Then $(\rho_1, t_1, u_1, \H_1)$ is a completely contractive  representation of $((X,\C), \G, \alpha)$ such that
\[
t_1(\xi)^*t_1(\eta) = \left[\begin{array}{cc} \rho(\langle \xi,\eta\rangle) & 0 \\ 0&0\end{array}\right]
\]
for all $\xi,\eta\in\H$.
\end{lemma}

\begin{proof}
By \cite[Lemma 3.7]{MS} $(\rho_1, t_1)$ is a completely contractive  representation of $(X,\C)$ on $\H_1$ which satisfies the last two statements in the lemma.

First note that $\tilde u$ is in fact well-defined since it respects the internal $\C$-modularity of $\H^X$,
\begin{align*}
\tilde u(s)(\xi c\otimes h) & = \alpha_s(\xi c) \otimes u(s)h
\\ & = \alpha_s(\xi)\alpha_s(c) \otimes u(s)h
\\ & = \alpha_s(\xi) \otimes \rho(\alpha_s(c))u(s)h
\\ & = \alpha_s(\xi) \otimes u(s)\rho(c)h
\\ & = \tilde u(s)(\xi \otimes \rho(c)h),
\end{align*}
for all $s\in \G, c\in \C, \xi\in X, h\in \H$.

Additionally, observe that $\tilde u(s)$ is unitary since for all $s\in \G, \xi,\eta\in X$ and $h,k\in \H$ we have that 
\begin{align*}
\langle \tilde u(s)(\xi\otimes h), \tilde u(s)(\eta\otimes k)\rangle & = \langle \alpha_s(\xi)\otimes u(s)h, \alpha_s(\eta)\otimes u(s)k\rangle
\\ & = \langle u(s)h, \rho(\langle \alpha_s(\xi),\alpha_s(\eta)\rangle)u(s)k\rangle
\\ & = \langle u(s)h, \rho(\alpha_s(\langle \xi,\eta\rangle))u(s)k\rangle
\\ & = \langle u(s)h, u(s)\rho(\langle \xi,\eta\rangle)k\rangle
\\ & = \langle h, \rho(\langle \xi,\eta\rangle)k\rangle
\\ & = \langle \xi\otimes h, \eta\otimes k\rangle.
\end{align*}

Next we need to make the following calculations:
\begin{align*}
\sigma^X(\alpha_s(\xi))u(s)h & = \alpha_s(\xi)\otimes u(s)h 
\\ & = \tilde u(s)(\xi \otimes h)
\\ & = \tilde u(s)\sigma^X(\xi)h
\end{align*}
and
\begin{align*}
\langle {\tilde t}^*\tilde t \sigma^X(\alpha_s(\xi))u(s)h, \eta\otimes k\rangle 
& = \langle \tilde t(\alpha_s(\xi)\otimes u(s)h), \tilde t(\eta\otimes k)\rangle
\\ & = \langle t(\alpha_s(\xi))u(s)h, t(\eta)k\rangle
\\ & = \langle u(s)t(\xi)h, t(\eta)k\rangle
\\ & = \langle t(\xi)h, t(\alpha_{s^{-1}}(\eta))u(s)^*k\rangle
\\ & = \langle \tilde t^* \tilde t(\xi\otimes h), \alpha_{s^{-1}}(\eta)\otimes u(s^{-1})k\rangle
\\ & = \langle \tilde u(s) \tilde t^* \tilde t(\xi\otimes h), \eta\otimes k\rangle
\\ & = \langle \tilde u(s) \tilde t^* \tilde t\sigma^X(\xi)h, \eta\otimes k\rangle.
\end{align*}
Combining these one gets 
\begin{align*}
\tilde u(s)(I - \tilde t^*\tilde t)\sigma^X(\xi) &= (I - \tilde t^*\tilde t)\sigma^X(\alpha_s(\xi))u(s)
\\ & = (I - \tilde t^* \tilde t)\tilde u(s) \sigma^X(\xi).
\end{align*}
Hence, 
\[
\tilde u(s)(I - \tilde t^*\tilde t) = (I - \tilde t^*\tilde t)\tilde u(s)
\]
and by a standard trick often attributed to Halmos
\[
\tilde u(s)(I - \tilde t^*\tilde t)^{1/2} = (I - \tilde t^*\tilde t)^{1/2}\tilde u(s).
\]

Now, we need to establish the covariance relations between $(\rho_1, t_1)$ and $u_1$.
From the previous paragraph we have that
\[
\tilde u(s)(I - \tilde t^*\tilde t)^{1/2}\sigma^X(\xi) = (I - \tilde t^*\tilde t)^{1/2}\sigma^X(\alpha_s(\xi))u(s)
\]
and thus $u_1(s)t_1(\xi) = t_1(\alpha_s(\xi))u_1(s)$.

Second, it is much more straightforward to calculate that
\begin{align*}
\tilde u(s) \tilde \rho(c)(\xi\otimes h) & = \tilde u(s)(\varphi_X(c)\xi \otimes h)
\\ & = \alpha_s(\varphi_X(c)\xi) \otimes u(s)h
\\ & = \varphi_X(\alpha_s(c))\alpha_s(\xi) \otimes u(s)h
\\ & = \tilde \rho(\alpha_s(c))\tilde u(s)(\xi \otimes h).
\end{align*}
Therefore, $u_1(s) \rho_1(c) = \rho_1(\alpha_s(c))u_1(s)$ and the conclusion follows.
\end{proof}

%%%%%
\begin{theorem}\label{thm::dilation}
Every completely contractive  representation of the non-degenerate $\ca$-correspondence dynamical system $((X,\C),\G,\alpha)$ has a minimal isometric coextension. Moreover, the minimal isometric coextension is unique up to unitary equivalence.
\end{theorem}

\begin{proof}
Let $(\rho,t,u,\H)$ be a completely contractive  representation of $((X,\C),\G,\alpha)$. Following the proof of \cite[Theorem 3.3]{MS} repeatedly use Lemma \ref{lemma::onestep} to get a sequence of completely contractive  representations $(\rho_n, t_n, u_n, \H_n)$ of $((X,\C), \G, \alpha)$ in the obvious manner: use $(\rho,t, u, \H)$ to produce $(\rho_1,t_1, u_1, \H_1)$ and then recursively $$(\rho_n,t_n, u_n, \H_n) = ((\rho_{n-1})_1, (t_{n-1})_1, (u_{n-1})_1, (\H_{n-1})_1)$$ from the previous lemma and the discussion preceding it.

Let $\H' = \overline{\cup_{n\geq 1} \H_n}$ and define $\rho' = \varinjlim \rho_n, t' = \varinjlim t_n$ and $u' = \varinjlim u_n$. By \cite[Theorem 3.3]{MS} $(\rho', t',\H')$ is an isometric  representation of $(X,\C)$. Note that $u'$ is a strongly continuous unitary representation of $\G$ since it is the direct sum of such representations.

Now to the covariance relations:
\begin{align*}
u'(s) \rho'(c)P_{\H_n}  & = u_n(s) \rho_n(c)P_{\H_n} 
\\ & = \rho_n(\alpha_s(c))u_n(s) P_{\H_n}
\\ & = \rho'(\alpha_s(c))u'(s) P_{\H_n}
\end{align*}
for all $s\in\G, c\in \C$ and $n\in \bbN$.
As well,
\begin{align*}
u'(s) t'(\xi) P_{\H_n} & = u_{n+1} t_{n+1}(\xi) P_{\H_n}
\\ & = t_{n+1}(\alpha_s(\xi))u_{n+1}(s) P_{\H_n}
\\ & = t_{n+1}(\alpha_s(\xi))u_n(s) P_{\H_n}
\\ & = t'(\alpha_s(\xi))u'(s)P_{\H_n}
\end{align*}
for all $s\in \G, \xi\in X$ and $n\in \bbN$. Therefore, the covariance relations are satisfied and $(\rho',t',u',\H')$ is an isometric coextension of $(\rho,t,u,\H)$.

Now let 
\[
\K = \overline{\spn}\{t'(\xi_1)\cdots t'(\xi_n)h : \xi_i\in X, h\in\H, n\geq 1\} \subset \H'.
\]
Because of the covariance relations of $(\rho',t',u',\H')$ one can see that $\K$ is a reducing subspace of $(\rho',t',u',\H')$ that contains $\H$. Thus, $(\rho', t', u', \K)$ is a minimal isometric coextension of $(\rho, t, u, \H)$.

In regard to uniqueness, suppose that $(\rho'', t'', u'', \H'')$ is another isometric coextension of $(\rho,t,u,\H)$. \cite[Proposition 3.2]{MS} proves that there exists a unitary $W: \K \rightarrow \H''$ such that $W\rho'(\cdot) = \rho''(\cdot)W$, $Wt'(\cdot) = t''(\cdot)W$ and $Wh = h$, for all $h\in \H$. Now
\begin{align*}
Wu'(s)t'(\xi_1)\cdots t'(\xi_n)h & = Wt'(\alpha_s(\xi_1))\cdots t'(\alpha_s(\xi_n))u'(s)h
\\ & = t''(\alpha_s(\xi_1)) \cdots t''(\alpha_s(\xi_n))Wu(s)h
\\ & = t''(\alpha_s(\xi_1)) \cdots t''(\alpha_s(\xi_n))u(s)h
\\ & = t''(\alpha_s(\xi_1)) \cdots t''(\alpha_s(\xi_n))u''(s)Wh
\\ & = u''(s)t''(\xi_1)\cdots t''(\xi_n)Wh
\\ & = u''(s)W t'(\xi_1)\cdots t'(\xi_n)h
\end{align*}
Therefore, by minimality $Wu'(\cdot) = u''(\cdot)W$ and so $(\rho',t',u',\K)$ and $(\rho'', t'', u'', \H'')$ are unitarily equivalent.
\end{proof}

%%%
\begin{theorem}\label{thm::isometrictensor}
Let $(X,\C)$ be a non-degenerate $\ca$-correspondence and let $\alpha$ be a generalized gauge action of a locally compact group $\G$. Then 
\[
\T_X^+ \rtimes_\alpha \G \simeq \T_X^+ \rtimes_{\T_X, \alpha} \G  \simeq \T_{X \rtimes_\alpha \G}^+
\]
\end{theorem}
\begin{proof}
It is already proven in \cite{KatRamMem}, in a discussion following Theorem 7.13, that $\T_X^+ \rtimes_{\T_X, \alpha} \G  \simeq \T_{X \rtimes_\alpha \G}^+$. (It also follows from \cite[Theorem 3.1]{BKQR} as the isomorphism $\Phi$ of that theorem maps generators to generators.)

Towards proving the remaining isomorphism, let $\varphi : \T^+_X \rtimes_\alpha \G \rightarrow B(\H)$ be a completely contractive representation. In the same way as in the proof of \cite[Theorem 4.1]{KatRamMem} one can assume that $\varphi$ is nondegenerate. Now by \cite[Proposition 3.8]{KatRamMem} there exists a  representation $(\pi, u, \H)$ of $(\T^+_X, \G, \alpha)$ so that $\varphi = \pi \rtimes u$.

By \cite[Theorem 3.10]{MS} there is a completely contractive  representation, $(\rho, t)$, of $(X,\C)$ such that $\pi = \rho \rtimes t$. Hence, in the same way as the first part of the proof of Theorem \ref{thm::isometricreps}, $(\rho, t, u, \H)$ is a completely contractive  representation of $((X,\C), \G, \alpha)$. By Theorem \ref{thm::dilation} $(\rho,t,u,\H)$ has a unique minimal isometric coextension $(\rho',t',u',\H')$ and thus by Theorem \ref{thm::isometricreps} $(\rho'\rtimes t', u', \H')$ is an isometric  representation of $(\T_X, \G, \alpha)$. 

As discussed in the proof of \cite[Theorem 3.10]{MS} $\H\subset \H'$ is seminvariant for $\rho'$ and $t'$ and thus $P_\H \rho'\rtimes t'|_\H$ is a completely contractive representation of $\T_X^+$. Moreover, that same theorem gives that 
\[
\pi = \rho\rtimes t = P_\H \rho'\rtimes t'|_\H
\]
because $\rho(c) = P_\H \rho'(c)|_\H$ and $t(\xi) = P_\H t'(\xi)|_\H$ for all $c\in\C$ and $\xi\in X$.

Therefore, every completely contractive representation of $\T_X^+ \rtimes_\alpha \G$ dilates to a completely contractive representation of $\T_X^+ \rtimes_{\T_X, \alpha} \G$ and thus they are completely isometrically isomorphic.
\end{proof}

\begin{corollary}
Let $((X, \A), \G,\alpha)$ be a $\ca$-correspondence dynamical system and assume that $\J_X=\{0 \}$.  Then all relative crossed products for $(\T_X^+, \G, \alpha)$ are canonically isomorphic via completely isometric maps.
\end{corollary}

In particular, the above applies to the non-commutative disc algebra $\A_{\infty}\subseteq \O_{\infty}$. To obtain the same conclusion for the non-commutative disc algebras $\A_n$, $n <\infty$, we need to work much harder. (See the next section.)

For the moment, we can put together all previous results to obtain the following, which summarizes our knowledge on the Hao-Ng isomorphism problem for the full crossed product.

\begin{theorem} \label{thm;summarize}
Let  $((X, \A), \G,\alpha)$ be a non-degenerate $\ca$-correspondence dynamical system. Then the following two statements are equivalent 
\begin{itemize}
\item[(i)] $\cenv (\T^+_X\cpf ) \simeq \O_X\cpf$ via a $*$-isomorphism that sends generators to generators,
\item[(ii)] $\O_X\cpf \simeq \O_{X\cpf}$ via a $*$-isomorphism that sends generators to generators, (Hao-Ng isomorphism)
\end{itemize}
and both imply
\begin{itemize}
\item[(iii)] all relative crossed products for $(\T^+_X, \G, \alpha)$ are completely isometrically isomorphic via canonical maps.
\end{itemize}
If $(X, \C)$ is hyperrigid, e.g., $\phi_X(\J_X)X=X$, then all of the above statements are equivalent.
\end{theorem}
\begin{proof}
Assume first that  
\begin{equation} \label{eq:main1}
\cenv (\T^+_X\cpf ) \simeq \O_X\cpf
\end{equation}
canonically. Theorem~\ref{thm::isometrictensor} shows now that
\[
\T_X^+ \rtimes_\alpha \G \simeq \T_X^+ \rtimes_{\T_X, \alpha} \G  \simeq \T_{X \rtimes_\alpha \G}^+
\]
canonically and so by taking $\ca$-envelopes 
%and invoking \cite[Theorem 3.7]{KatsoulisKribsJFA} 
we have a canonical isomorphism
\begin{equation} \label{eq:main2}
\cenv\big(\T_X^+ \rtimes_\alpha \G\big) \simeq \cenv\big(\T_X^+ \rtimes_{\T_X, \alpha} \G\big)  \simeq \O_{X \rtimes_\alpha \G}.
\end{equation}
By ``equating" the right sides of (\ref{eq:main1}) and (\ref{eq:main2}), we obtain (ii).

Conversely, assume that (ii) holds. Then by taking $\ca$-envelopes in the isomorphisms of Theorem~\ref{thm::isometrictensor} we obtain
\[
\cenv(\T_X^+ \rtimes_\alpha \G ) \simeq \cenv(\T_X^+ \rtimes_{\T_X, \alpha} \G)  \simeq \cenv(\T_{X \rtimes_\alpha \G}^+)\simeq \O_{X \rtimes_\alpha \G} \simeq \O_X\cpf
\]
by (ii). Therefore (i) holds, as desired.

Assume now that (ii) is valid, i.e., the Hao-Ng isomorphism is implemented via a canonical map. The same map establishes
\begin{equation} \label{eq:main3}
\T^+_X \rtimes_{\O_X, \alpha} \G \simeq \T^+_{X\cpf}.
\end{equation}
By Theorem~\ref{thm::isometrictensor} we also have
\begin{equation} \label{eq:main4}
\T^+_X\cpf \simeq \T^+_{X\cpf}.
\end{equation}
From (\ref{eq:main3}) and (\ref{eq:main4}), we obtain $\T^+_X\cpf \simeq \T^+_X \rtimes_{\O_X, \alpha} \G$, or,
\[
\T^+_X\cpf \simeq \T^+_X \rtimes_{\cenv(\T^+_X), \alpha} \G
\]
canonically. By Theorem~\ref{thm::minmax}, all relative crossed products for $(\T^+_X, \G, \alpha)$ are canonically isomorphic, which is (iii).

Assume now that $(X, \C)$ is hyperrigid and all relative crossed products for $(\T^+_X, \G, \alpha)$ are canonically isomorphic. Therefore (\ref{fullequiv}) in Theorem \ref{hyperenv} implies
\[
 \cenv\big(\T^+_X\rtimes_{\O_X, \alpha} \G \big) \simeq \O_X \cpf.
 \]
 By assumption $\T^+_X\rtimes_{\O_X, \alpha} \G\simeq \T^+_X \cpf$ and so (i) is valid.
 
 Finally recall that Theorem~\ref{thm:hyperrigid} shows that a $\ca$-correspondence $X$ with $\phi_X(\J_X)X=X$ is always hyperrigid.
\end{proof}

The importance of the previous result can not be understated. First, note that condition (i) in Theorem~\ref{thm;summarize} is just the equivalence 
\[
\cenv(\A\cpf)\simeq \cenv(\A) \cpf
\]
of \cite[Problem 1]{KatRamMem} with $\A=\T^+_X$. Hence the equivalence of (i) and (ii) in Theorem~\ref{thm;summarize} shows that the Hao-Ng isomorphism and Problem 1 in \cite{KatRamMem} are actually equivalent problems in the context of tensor algebras of $\ca$-correspondences. Furthermore, if the Hao-Ng isomorphism holds, then we automatically have from (iii) that $\O_{X\cpf}\simeq \O_{X\cpd}$. Therefore a positive resolution for the Hao-Ng isomorphism conjecture also implies a positive resolution for the modified conjecture of \cite[page 70]{KatRamMem}.

Also notice that according to condition (iii), the verification of the Hao-Ng isomorphism for hyperrigid $\ca$-correspondences depends on the canonical identification of two non-selfadjoint operator algebras. We pursue this direction successfully in the next section where we verify the Hao-Ng isomorphism for all graph correspondences of row finite graphs. Furthermore,  unlike condition (ii) (Hao-Ng isomorphism), both conditions (i) and (iii) are applicable to \textit{arbitrary} dynamical systems $(\T^+_X, \G, \alpha)$, i.e., $\alpha$ does not have to be a gauge action. Thus in a sense, a generalization of the Hao-Ng isomorphism problem beyond the realm of gauge actions is possible but only in the language of non-selfadjoint operator algebras. In light of the recent results of Harris and Kim \cite{HK}, this seems to be a direction worth pursuing.

%%%%%%%%%
\section{Graph correspondences}

Following from the last section, we would like to prove that every isometric (Toeplitz) representation of $((X,\C),\G,\alpha)$ dilates to a covariant (Cuntz-Pimsner) representation. However, the standard proofs that the C$^*$-envelope of the tensor algebra is the Cuntz-Pimsner algebra are non-constructive \cite{KatsoulisKribsJFA, MS} which at the moment is a barrier to our method of proof. Significantly. in the case of graph correspondences such a constructive dilation proof is shown to exist.

Let $(E, V, s, r)$ be a directed graph, where both $E$ and $V$ are separable, with associated graph correspondence $(X, \C, \varphi_X)$. Recall this is where $\C = c_0(V)$, $X$ is the completion of $c_c(E)$ under the right module structure
\begin{align*}
\langle \alpha\delta_e, \beta\delta_f\rangle & = \left\{ \begin{array}{ll} \overline\alpha \beta\delta_{s(e)}, & e = f
\\ 0, & \textrm{otherwise}\end{array}\right.
\\ \delta_e\cdot \delta_v & = \left\{ \begin{array}{ll} \delta_e, & s(e) = v
\\ 0, & \textrm{otherwise} \end{array}\right. ,
\end{align*}
and the left action of $\C$ on $X$ is given by
\[
\varphi_X(\delta_v) \delta_e = \left\{\begin{array}{ll} \delta_e, & r(e) = v \\ 0, & \textrm{otherwise} \end{array}\right. .
\]
When $(X,\C)$ is the graph correspondence for a directed graph $(V,E)$ then $\O_X$ is $*$-isomorphic to the Cuntz-Krieger algebra of the graph. 

%To this end, recall that given a C$^*$-correspondence $(X,\C)$ for $\xi,\eta\in X$ we denote by $\theta_{\xi,\eta} \in \L(X)$ the adjointable operator $\theta_{\xi,\eta}(\zeta) = \xi\langle \eta,\zeta\rangle, \zeta\in X$ and call $\K(X) = \overline{\spn}\{\theta_{\xi,\eta} : \xi,\eta\in X\}$ the ideal of compact operators. Note that these operators are usually not compact in the traditional sense but rather are defined very similarly to rank one operators.

%Suppose $(\rho, t, \H)$ is an isometric representation of $(X,\C)$. There is a $*$-homomomorphism $\psi_t : \K(X) \rightarrow B(\H)$ given by $\psi_t(\theta_{\xi,\eta}) = t(\xi)t(\eta)^*$.
%Now $(\rho, t, \H)$ is called a {\em Cuntz-Pimsner representation} (some sources call this a fully coisometric or covariant representation) of $(X,\C)$ if $\psi_t(\varphi_X(a)) = \rho(a)$, for every $a\in \\J_X$, where
%\[
%\\J_X = \varphi_X^{-1}(\K(X)) \cap (\ker \varphi_X)^\perp.
%\]
%The C$^*$-algebra generated by the universal Cuntz-Pimsner representation is the Cuntz-Pimsner algebra $\O_X$. 

We wish to find Cuntz-Pimsner representations of these graph correspondences. 
As usual, the main concern is looking at which elements of $\C$ are mapped into $\K(X)$ by $\varphi_X$.

In the case of a graph correspondence, Raeburn \cite[Proposition 8.8]{Raeburn} gives that $\varphi_X(\delta_v) \in \K(X)$ if and only if $|r^{-1}(v)| < \infty$. Furthermore, $\delta_v \in \ker \phi_X$ if and only if $r^{-1}(v) = \emptyset$.
Thus, let 
\[
V_{\fin}= \{v\in V : 1 \leq |r^{-1}(v)| < \infty\}
\]
 be the set of vertices generating Katsura's ideal $\J_X$ and let
 \[
 \K=\{ (v,w)\mid  \exists e \in E \mbox{ with }s(e)=w, r(e) =v, v \in V_{\fin}  \}.
 \]
 
 For each pair $(v,w) \in \K$, let $E(v,w)$ be the collection of all edges starting from $w$ and ending on $v$ and let $[E((v, w)]= \bbC^{|E(v, w)|}$. In what follows we will identify the canonical basis of $[E(v,w)]$ with the elements of $E(v,w)$ and use the same symbol for both.
 
Suppose $\alpha$ is a generalized gauge action of $(X,\C)$ then it is clear that this induces a permutation of $V$ and in particular that $V_{\fin}$ is invariant under this permutation. By abuse of notation we call this permutation $\alpha : V \rightarrow V$. Furthermore, the action $\alpha$ maps $[E(v,w)]$ unitarily onto $[E(\alpha(v), \alpha(w))]$. Indeed, if $E(v, w)= \{e_1, e_2, \dots, e_n\}$ and $\xi= \sum_{i=1}^{n}c_i \delta_{e_i}$, then
\begin{align*}
\sca{\alpha(\xi), \alpha(\xi)}&= \Big\langle \alpha\Big(\sum_{i=1}^n c_i \delta_{e_i}\Big), \alpha\Big(\sum_{j=1}^n d_e \delta_{e_j}\Big)\Big\rangle \\
& = \alpha \Big(\Big\langle \sum_{i=1}^n c_i \delta_{e_i}, \sum_{j=1}^n c_j \delta_{e_j} \Big\rangle\Big)
\\ & = \alpha\Big( \sum_{i=1}^n |c_i|^2 \delta_{w} \Big) \\
& = \sum_{i=1}^n |c_i|^2 \delta_{\alpha(w)} = \sca{\xi, \xi}.
\end{align*}

%%%%%%%%%%
\begin{proposition}
Let $(X, \C)$ be the graph correspondence of $(E, V)$ and suppose $(\rho,t,u,\H)$ is a completely contractive representation of the dynamical system $((X,\C),\G, \alpha)$. There exists a dilation to a completely contractive  representation $(\rho_1,t_1,u_1,\H_1)$ such that for every $v\in V_{\fin}$ 
\[
\rho(\delta_v) = \sum_{e\in r^{-1}(v)} t_1(\delta_e)t_1(\delta_e)^* 
\]
\end{proposition}
\begin{proof}
By \cite[Lemma 3.5]{MS} because $(\rho, t,\H)$ is completely contractive then for $v\in V_\K$ we have the matrix inequality 
\[
[t(e)^*t(f)]_{e, f\in r^{-1}(v)} \leq [\rho(\langle e,f\rangle]_{e,f\in r^{-1}(v)} = \oplus_{e\in r^{-1}(v)} \rho(\delta_{s(e)})
\]
and so $[t(e) : e\in r^{-1}(v)]$ is a row contraction. Hence,
\[
\rho(\delta_v) \geq  \sum_{e\in r^{-1}(v)} t(\delta_e)t(\delta_e)^*
\]
and so we can define $$\Delta_v := \frac{1}{\sqrt{|r^{-1}(v)|}}\Big(\rho(\delta_v) -  \sum_{e\in r^{-1}(v)} t(\delta_e)t(\delta_e)^*\Big)^{1/2}.$$

Let $\H_w= \rho(\delta_w)\H$ and so we can assume $\H = \oplus_{w\in V} \H_w$. For each pair $(v',w') \in \K$ let
\[
\H_{v',w'} :=\H_{v'}\otimes  [E(v',w')]. 
\]
Then for each $w \in V$ we define
\[
\H^+_w\equiv \bigoplus_{(v', w)\in \K} \H_{v',w}=\bigoplus_{(v', w)\in \K} \H_{v'}\otimes  [E(v',w)].
\]
and
\[
\H_{w, 1} = \H_{w} \oplus\H^+_w.
\]
We combine these to define 
\[
\H_1 = \oplus_{w\in V} \H_{w, 1}.
\] 
Hence
\[
\rho_1(\delta_w) = I_{\H_{w,1}}, \quad w\in V,
\]
extends to a $*$-homomorphism of $\C$ that dilates $\rho$.

We also have a continuous unitary representation $u_1:\G \rightarrow B(\H_1)$ dilating $u: \G \rightarrow B(\H)$ and defined as follows.

Given $g \in \G$ and $h \in \H$, we let $u_1(g)h=u(g)h$. Otherwise, on each $\H_w^+$ the operator $u_1(g)$ is defined by
\[
\H_w^+\supseteq \H_{v',w}\ni h\otimes \xi\longmapsto u_g(h)\otimes\alpha_g(\xi)\in \H_{\alpha_g(v'),\alpha_g(w)}\subseteq \H_{\alpha(w)}^+.
\]
It is easy to see that $u_1:\G \rightarrow \B(\H_1)$ is a continuous unitary representation dilating $u$.

We are ready to dilate $t:X\rightarrow B(\H)$. If $r(e) \notin V_{\fin}$, then we let $t_1(e)=t(e)$. Otherwise, if $e \in E$ with $s(e)=w$ and $r(e)=v \in V_{\fin}$, then $t_1(e) \in B(\H_1)$ has cokernel contained in $$\H_w\oplus (\dots 0\oplus 0 \oplus \H_{v,w}\oplus0 \dots)\subseteq \H_w \oplus\H^+_w\equiv \H_{w,1}$$ range contained in $\H_v\oplus 0\subseteq \H_v\oplus \H^+_v$ and it is given by  
\[
t_1(\delta_e) = \left[t(\delta_e)  \ \ \Delta_v \tau(e)\right] \in B(\H_w\oplus \H_{v,w}, \H_v) ,
\]
where 
\[
\tau(e)\colon \H_{v,w}\longrightarrow \H_v \,  ;\, h\otimes \xi \longmapsto \sca{e,\xi}h.
\]
(In general, for $\zeta \in [E(v,w)]$, $\tau(\zeta)$ will be given by $\tau(\zeta)(h \otimes \xi)=\sca{\zeta,\xi}h$.)

It is easy to see that $\tau(e)\tau(e)^*=I_{\H_v}$ and so,
\begin{align*}
\sum_{e\in r^{-1}(v)} t_1(\delta_e)t_1(\delta_e)^*
& = \sum_{e\in r^{-1}(v)} t(\delta_e)t(\delta_e)^* + \Delta_v \tau(e)\tau(e)^*\Delta_v
\\ & = \Big( \sum_{e\in r^{-1}(v)} t(\delta_e)t(\delta_e)^* \Big)+  |r^{-1}(v)|\Delta_v^2
\\ & = \rho(\delta_v)
\end{align*}
This establishes one of the main conclusions of this proposition and gives that
\[
[t_1(e)^*t_1(f)]_{e,f\in r^{-1}(v)} \leq [\rho_1(\langle e,f\rangle)]_{e,f\in r^{-1}(v)}
\]
as in the start of the proof. Thus, by \cite[Lemma 3.5]{MS} again, $(\rho_1,t_1,\H_1)$ is a completely contractive representation of $(X,\C)$.
\vskip 6 pt

Lastly we must establish the covariance relations. For any $g\in \G$ recall that $\alpha_g$ acts as a unitary between $[E(v,w)]=[\{e_1, e_2, \dots, e_n\}]$ and $[E(\alpha_g(v),\alpha_g(w))]=[\{f_1,f_2, \dots ,f_n\}]$. This implies that
\begin{align*}
\sum_{i=1}^n t(\alpha_g(\delta_{e_i}))t(\alpha_g(\delta_{e_i}))^* &= \sum_{i=1}^n t\big( \sum_{j=1}^n (\alpha_g)_{j,i} \delta_{f_j} \big) t\big( \sum_{k=1}^n (\alpha_g)_{k,i} \delta_{f_k}\big)^*
\\ & =  \sum_{j,k=1}^n \left(\sum_{i=1}^n (\alpha_g)_{j,i}\overline{(\alpha_g)}_{k,i}\right)t(\delta_{f_j})t(\delta_{f_k})^*
\\ & = \sum_{j=1}^n t(\delta_{f_j})t(\delta_{f_j})^*
\end{align*}
and so
\begin{align*}
u(g)(\rho(\delta_v) \ -  &\sum_{e\in r^{-1}(v)} t(\delta_e)t(\delta_e)^*)
\\ & = \big(\rho(\delta_{\alpha_g(v)}) - \sum_{e\in r^{-1}(v)} t(\alpha_g(\delta_e))t(\alpha_g(\delta_e))^*\big)u(g)
\\ & = \Big(\rho(\delta_{\alpha_g(v)}) - \sum_{w \in s(r^{-1}(v))}\, \sum_{e\in E(v,w) }t(\alpha_g(\delta_e))t(\alpha_g(\delta_e))^*\Big)u(g)
\\ & = \Big(\rho(\delta_{\alpha_g(v)}) - \sum_{w \in s(r^{-1}(\alpha_g(v)))}\,\,  \sum_{f\in E(\alpha_g(v), w)} t(\delta_{f})t(\delta_{f})^*\Big)u(g)
\\ & = \Big(\rho(\delta_{\alpha_g(v)}) - \sum_{f\in r^{-1}(\alpha_g(v))} t(\delta_f)t(\delta_f)^*\Big) u(g).
\end{align*}
By a standard functional analysis trick,
\[
u(g) \Delta_v = \Delta_{\alpha_g(v)}u(g).
\]
Furthermore, if $h\otimes\xi \in \H_{v,w}$, then
\begin{align*}
 \tau(\alpha_g(e))u_1(g)(h\otimes\xi) &=\tau(\alpha_g(e))(u(g)h\otimes \alpha_g(\xi)\\
&=\sca{\alpha_g(e), \alpha_g(\xi)} u(g)h\\
&=\sca{e, \xi}u(g)h =u(g)\tau(e)(h\otimes\xi).
\end{align*}
Hence,
\begin{align*}
t_1(\alpha_g(\delta_e))u_1(g)
&= \left[ t(\alpha_g(\delta_e)) \ \ \Delta_{\alpha_g(v)}\tau(\alpha_g(e))\right] \big(u_1(g)\mid_{\H_w\oplus\H^+_w}\big)\\
&= \left[u(g)t(\delta_e) \ \ \Delta_{\alpha_g(v)} u(g)\tau(e)\right] \\
&= \left[u(g)t(\delta_e) \ \ u(g) \Delta_vu(g)\tau(e)\right] \\
&= u_1(g) t_1(\delta_e).
\end{align*}
It is also immediate that
\[
u_1(g)\rho_1(\delta_v) = \rho_1(\delta_{\alpha_g(v)})u_1(g).
\]
Therefore, $(\rho_1,t_1,u_1,\H_1)$ is a completely contractive  representation of $((X,\C),\G,\alpha)$.
\end{proof}

%%%%%%%%
\begin{theorem}\label{thm::CPdilation}
Let $(X,\C)$ be the graph correspondence of $(E,V)$. Every completely contractive representation of $((X,\C), \G, \alpha)$ can be dilated to a Cuntz-Pimsner  representation. 
\end{theorem}
\begin{proof}
Let $(\rho_0,t_0,u_0,\H_0)$ be a completely contractive representation of $((X,\C),\G,\alpha)$. Recursively use the previous proposition to generate a sequence of completely contractive  representations $(\rho_n, t_n, u_n, \H_n)$ such that for each $n\geq 1$, 
$\H_{n-1} \subset \H_n$, $(\rho_n,t_n, u_n, \H_n)$ is a dilation of $(\rho_{n-1},t_{n-1},u_{n-1},\H_{n-1})$ and for every $v\in V$ such that $1\leq |r^{-1}(v)| < \infty$ we have
\[
\sum_{e\in r^{-1}(v)} t_n(\delta_e)t_n(\delta_e)^* = \rho_{n-1}(\delta_v).
\]
Thus, define $\H' = \cup_{n=0}^\infty \H_n$ and 
\[
\rho'(c)|_{\H_n} = \rho_n(c), t'(\xi)|_{\H_n} = t_n(\xi), \ \textrm{and} \ \ u'(g)|_{\H_n} = u_n(g).
\]
Hence, $(\rho', t', u', \H')$ is a dilation $(\rho_0, t_0, u_0, \H_0)$ to a completely contractive  representation such that for every $v\in V_{\fin}$
\begin{align*}
\sum_{e\in r^{-1}(v)} t'(\delta_e)t'(\delta_e)^* & = \sot-\lim_{n\rightarrow \infty} \sum_{e\in r^{-1}(v)} t'(\delta_e)I_{\H_n}t'(\delta_e)^*
\\ & = \sot-\lim_{n\rightarrow \infty} \sum_{e\in r^{-1}(v)} t_n(\delta_e)t_n(\delta_e)^*
\\ & = \sot-\lim_{n\rightarrow \infty} \rho_{n-1}(\delta_v)
\\ & = \rho'(\delta_v).
\end{align*}
According to \cite[Definition 5.3]{MS} the representation $(\rho', t', \H')$ is $\J_X$-coisometric, i.e., Cuntz-Pimsner in the sense of Katsura but without being isometric.

Lastly, by Theorem \ref{thm::dilation} there is a unique minimal isometric coextension of $(\rho',t',u',\H')$ to $(\rho'', t'', u'', \H'')$. By \cite[Corollary 5.21]{MS} this coextension process preserves the property of being $\J_X$-coisometric. Therefore, $(\rho'', t'', u'', \H'')$ is an isometric and $\J_X$-coisometric dilation of $(\rho_0,t_0,u_0,\H_0)$, that is, a Cuntz-Pimsner representation of $((X,\C),\G, \alpha)$.
\end{proof}

%%%%%%%%
\begin{corollary}
Let $(X,\C)$ be the graph correspondence of a directed graph $(E,V)$ and $((X,\C), \G, \alpha)$ a $\ca$-correspondence dynamical system, that is,  $\alpha$ is a generalized gauge action. Then all relative crossed products for $(\T_X^+, \G, \alpha)$ are canonically isomorphic via completely isometric maps.
\end{corollary}
\begin{proof}
This is an immediate consequence of Theorem~\ref{thm::CPdilation} and Theorem~\ref{thm::minmax}.\end{proof}

With an added assumption this gives a positive solution to the Hao-Ng isomorphism problem.

%%%%%%%%
\begin{corollary}\label{hao-ng}
Let $(X,\C)$ be the graph correspondence of a row-finite directed graph $(E,V)$ and $((X,\C), \G, \alpha)$ a $\ca$-correspondence dynamical system. Then
\[
\O_{X\rtimes_\alpha \G} \simeq \O_X \rtimes_\alpha \G.
\]
\end{corollary}
\begin{proof}
By the description of $\J_X$ mentioned at the start of this section and Theorem~\ref{thm:hyperrigid}, the graph correspondence of a row-finite directed graph $(E,V)$ is hyperrigid. The conclusion follows from Theorem~\ref{thm;summarize}.
\end{proof}

\vspace{0.1in}

{\noindent}{\it Acknowledgement.} Part of this research was carried out in the summer of 2018 while EK was visiting Chongqing Normal University, Chongqing University, the Chinese Academy of Science in Beijing and the Chern Institute of Mathematics. EK would like to thank his hosts Professors Liming Ge, Hanfeng Li, Chi-Keung Ng, Wenhua Qian, Wenmimg Wu and Wei Yuan for the stimulating conversations and their hospitality during his stay at their institutions.
\vspace{0.1in}
%%%%%%%%%%%%%%%%%%%%%%%%%%%%%%%%

\end{document}